\documentclass[11pt]{article}
\usepackage{amsmath,amsfonts,xypic}
\usepackage{xcolor}
\usepackage{mathtools}
\usepackage{booktabs}
\usepackage{amsthm}
\usepackage{amssymb}
\usepackage[all]{xy}
\usepackage{tikz}
\usepackage{pgfplots}
\usepackage{pstricks,pst-slpe}
\usepackage{pst-solides3d}
\usepackage{comment}
\usepackage{ulem}
\usepackage{enumitem}
\usepackage{rotating}
\usepackage{graphicx}
\usepackage[bottom]{footmisc}
\usepackage[font=small,labelfont=bf]{caption} 
\usepackage{authblk}
\advance \textheight by \topskip
\makeatletter
\@addtoreset{equation}{section}
 
\makeatother

\newtheorem{theorem}{Theorem}[section]

\newtheorem{corollary}[theorem]{Corollary}

\newtheorem{definition}[theorem]{Definition}
\newtheorem{example}[theorem]{Example}

\newtheorem{lemma}[theorem]{Lemma}

\newtheorem{proposition}[theorem]{Proposition}
\newtheorem{remark}[theorem]{Remark}

\def\ker{{\mathrm{Ker}}}
\def\nn{\nonumber}
\newcommand{\beq}{\begin{equation}}
\newcommand{\eeq}{\end{equation}}
\newcommand{\beqa}{\begin{eqnarray}}
\newcommand{\eeqa}{\end{eqnarray}}
\newcommand{\noi}{\noindent}

 \newcommand{\bpm}{\begin{pmatrix}}
 \newcommand{\epm}{\end{pmatrix}}

\def\>{\rangle}
\def\<{\langle}

\usepackage[colorlinks=true,urlcolor=black]{hyperref}
\hypersetup{			
backref = true,			
pagebackref = true,		
hyperindex = true, 		
colorlinks = true, 		
breaklinks = true, 		
urlcolor = blue, 		
linkcolor = blue, 		
bookmarks = true,		
bookmarksopen = true,
citecolor=red,
}
\usepackage{orcidlink}

\newcommand{\Cork}{\operatorname{Cork}}

\newcommand{\Spec}{\operatorname{Spec}}
\newcommand{\mult}{\operatorname{mult}}
\newcommand{\Adj}{\operatorname{Adj}}
\newcommand{\diag}{\operatorname{diag}}

\begin{document}

\title{{\bf On a new class of high-corank Kac--Moody algebras}}
\author[1]{S. Beaudoin${}^{\orcidlink{0009-0003-5612-4870}}$ \thanks{simon.beaudoin@iphc.cnrs.fr}}
\author[1]{Q. Bonnefoy${}^{\orcidlink{0000-0002-1102-5209}}$ \thanks{quentin.bonnefoy@iphc.cnrs.fr}}
\author[2]{A. Marrani${}^{\orcidlink{0000-0002-7597-1050}}$ \thanks{a.marrani@herts.ac.uk}}
\author[1]{M. Rausch de Traubenberg${}^{\orcidlink{0000-0001-5045-2353}}$ \thanks{michel.rausch@iphc.cnrs.fr}}
\author[1]{V. Saulquin${}^{\orcidlink{0009-0002-2175-6109}}$ \thanks{victor.saulquin@iphc.cnrs.fr}}
\affil[1]{Universit\'e de Strasbourg, CNRS, IPHC UMR7178, F-67037 Strasbourg Cedex, France}
\affil[2]{Centre for Mathematics and Theoretical Physics, University of Hertfordshire, UK}
\date{}
\maketitle

\bigskip
\bigskip
\bigskip

\abstract{We present recursive constructions of several families of generalized Cartan matrices associated with Kac--Moody algebras, whose sizes and coranks grow exponentially. The constructions are encoded by connected multigraphs and by block-doubling operations on their associated symmetric generalized Cartan matrices. Equivalently, the corank problem is translated into a spectral graph-theoretic problem: the corank of $2\mathrm{Id}-\operatorname{Adj}(G)$ is the multiplicity of the adjacency eigenvalue $2$. We give two explicit recursive families, compute their spectra and coranks, and emphasize the difference between absolute exponential growth and relative asymptotic density. The resulting algebras are typically indefinite and singular of corank larger than one, and therefore contain several independent central directions and several isotropic radical directions in the root lattice. We also discuss alternative constructions and possible applications to the algebraic structures appearing in gravity, supergravity, string/M-theory and related generalized symmetry problems.}

\setlength\parindent{0pt}

\bigskip

\newpage

\tableofcontents

\newpage

\section{Introduction}\label{sec:intro}

The purpose of this paper is to construct and analyse explicit infinite families of singular generalized Cartan matrices of large corank, together with their associated Kac--Moody algebras. The starting point is elementary: if $G$ is a connected undirected multigraph with no loops, then
\[
        A_G=2\mathrm{Id}-\operatorname{Adj}(G)
\]
is a symmetric indecomposable generalized Cartan matrix. Hence the corank of $A_G$ is the multiplicity of the adjacency eigenvalue $2$. The problem of producing high-corank Cartan matrices can therefore be interpreted as the problem of producing connected multigraphs in which the spectral value $2$ occurs with large multiplicity. This simple observation is one of the guiding principles of the constructions below.

Kac--Moody algebras were introduced as a vast extension of finite-dimensional semisimple Lie algebras, and their construction is governed by generalized Cartan matrices \cite{Kac,Moody,Kac2,moL}. The finite and affine cases are highly structured: finite type is characterized by positive-definite Cartan matrices, while affine type is characterized, in the indecomposable symmetrizable case, by positive semidefinite Cartan matrices of corank one. In affine types, the unique primitive positive null-vector gives the null root, and the corank-one singularity is reflected in the appearance of one central element and one degree derivation in the standard realization. Indefinite Kac--Moody algebras, by contrast, form a far broader and less rigid class. Within this class, singular generalized Cartan matrices of corank larger than one are natural objects, but explicit and controllable families of such matrices are comparatively less familiar.

The matrices constructed in this paper are indefinite and singular, with coranks that grow rapidly with the recursive steps. More precisely, their sizes grow exponentially, and the coranks grow exponentially along infinite subsequences. At the same time, the ratio between corank and size has a subtler asymptotic behavior, typically governed by central-binomial-type estimates. Thus ``high corank'' should be understood as an absolute large-rank phenomenon rather than, necessarily, as a positive-density phenomenon in the large-size limit.

A distinctive feature of the present construction is its compatibility with an explicit spectral recursion. Given a seed matrix $A$, we consider matrices of the form
\[
        A'=
        \begin{pmatrix}
        A&-L \ \mathrm{Id}\\
        -L \ \mathrm{Id}&A
        \end{pmatrix},
        \qquad L\in \mathbb{N}\setminus\{0\}
\]
The spectrum of $A'$ is obtained from the spectrum of $A$ by shifting by $\pm L$. Consequently, the zero-eigenspace of $A'$ is determined by the $\pm L$ eigenspaces of $A$. This is the algebraic mechanism behind the recursive growth of the corank. Geometrically, this block operation corresponds to taking two matched copies of a graph and connecting corresponding vertices by $L$ parallel edges, namely block-doubling. The resulting diagrams are therefore regular multigraphs closely related to hypercubes and modified hypercubes.

The first family starts from a modified cube construction whose third step gives a twelve-dimensional Cartan matrix of corank two. Subsequent steps are obtained by the uniform doubling procedure described above. The second family starts from the one-vertex graph and produces weighted hypercubes. In both cases the spectra can be computed in closed form. The first family shows rich dependence on the linking number $L$, including monotone and oscillatory corank patterns. The second family is more rigid and displays an alternating vanishing/non-vanishing behavior for the relevant values of $L$.

Beyond the explicit calculations, the paper is also meant to isolate several conceptual points. First, for an indecomposable Cartan matrix of corank greater than one, the kernel cannot contain a non-zero vector with all entries of the same sign. Consequently, non-zero integral kernel vectors give isotropic elements of the root lattice which are not roots, because their coordinates in the simple-root basis necessarily have mixed sign. Second, the radical of the symmetrized bilinear form is no longer one-dimensional. This distinguishes the present situation from the affine case and suggests that high-corank Kac--Moody algebras should be treated as genuinely multi-null objects. Third, the constructions sit naturally at the interface between Kac--Moody theory and spectral graph theory, and can therefore be generalized using graph products, lifts, covers, prescribed-kernel block matrices and reflection-geometric Gram matrices \cite{BrouwerHaemers,CvetkovicRowlinsonSimic,GodsilRoyle,BiluLinial,Vinberg67}.

There is also a broader motivation coming from theoretical and mathematical physics. Infinite-dimensional Kac--Moody algebras, especially affine, hyperbolic and Lorentzian ones, have appeared in two-dimensional integrable reductions of gravity, cosmological billiards, hidden symmetries of supergravity, and conjectural formulations of M-theory based on very-extended algebras such as $E_{10}$ and $E_{11}$ \cite{Julia85,DamourHenneauxNicolai2002,DamourHenneauxNicolai2003,WestE11,KleinschmidtNicolai2004,KleinschmidtNicolai2005}. The present high-corank families are not proposed here as immediate replacements for $E_{10}$ or $E_{11}$. Rather, they provide a controlled laboratory in which one can study the algebraic consequences of having many independent null (or zero-norm) and central directions. This may be relevant, at least speculatively, to singular limits of duality algebras, to multi-parameter degenerations, to tensor-hierarchy-like structures, and to situations in which several independent charges, constraints or null directions coexist.\\

The paper is organized as follows. Section~\ref{sec:defs} recalls the definitions of generalized Cartan matrices, the associated Kac--Moody algebras, the graph-theoretic language used throughout the paper, and the determinant identity underlying the spectral recursion. Section~\ref{sec:kernels} discusses the kernel of indecomposable Cartan matrices, emphasizing the distinction between the affine corank-one case and the high-corank case. Section~\ref{sec:constructions} gives the two explicit recursive constructions and derives their characteristic polynomials, spectra, degeneracies and coranks. Section~\ref{sec:further-methods} records additional methods for producing high-corank Cartan matrices beyond the two families studied in detail. Section~\ref{sec:conclusion} concludes with possible mathematical and physical developments. The appendices discuss the explicit corank-two example associated with $A_3$ and the diagram automorphisms of the constructed matrices.

\section{Some definitions}\label{sec:defs}

We begin with several definitions and constructions used throughout this note. We start with Kac--Moody algebras and their associated Cartan matrices, then we introduce the structures of connected graphs and cubes that we will base our construction on, as well as a property on the determinant of certain block matrices.

\subsection{Kac--Moody algebras and Cartan matrices}

Throughout this paper, we restrict ourselves to symmetrizable generalized Cartan matrices.
\begin{definition}\label{def:KM}
Let $A$ be an $r\times r$ matrix with integer entries, $A$ is a generalized Cartan matrix if (i) $A$ is
symmetrizable, {\it i.e.} there exists a diagonal matrix $D$ with positive coefficients such that $DA$ is a symmetric matrix, and (ii) $A$
satisfies the following:
\beqa
&A_{ij}\in \mathbb Z\nn\\
&A_{ii}=2\nn\\
&A_{ij}\leq0\ , \ \  i\ne j\nn\\
&A_{ij}=0 \ \Longleftrightarrow\  A_{ji}=0\nn
\eeqa
\end{definition}
From now on, we refer to generalized Cartan matrices simply as Cartan matrices, whenever no confusion is possible.\\

Let $s=$ Cork$(A)\equiv r$ $-$ rk$(A)$. Consider $3r+s$ linearly independent variables we denote $\{ h_a,e_{i},f_{i}: 1\le a\le r+s,1\le i\le r\}$, and set
\beqa
\mathfrak{h}'=\bigoplus\limits_{i=1}^r\mathbb C h_i\ , \qquad \mathfrak{h}=\bigoplus_{a=1}^{r+s}\mathbb C h_a\ , \qquad {\bf  r_+}=\bigoplus\limits_{i=1}^{r}\mathbb  C e_i \ , \qquad {\bf  r_-}=\bigoplus_{i=1}^{r}\mathbb C f_{i}\nn
\eeqa
Suppose $\alpha_1,\cdots ,\alpha_r\in\mathfrak{h}^*$ are linearly
independent and satisfy
\begin{equation}\label{alpha}
\alpha_j(h_i)=A_{ij}\ ,\quad  1\le i, \, j\le r\
\end{equation}
Then the center $Z=\text{Ker}(\alpha_1) \cap\cdots \cap\text{Ker}(\alpha_r)$ is of
dimension $s$, is contained in $\mathfrak{h}'$ and is independent of the choice of $\alpha_1,\cdots,\alpha_r\in\mathfrak{h}^*$ satisfying \eqref{alpha}. To the above data one can associate the following Lie algebra \cite{Kac,Moody}
(see also \cite{Kac2,moL,Mdo,Ca,Mar}):
\\
\begin{definition} \label{theo:Weyl}
The Kac--Moody algebra $\hat{\mathfrak{g}}(A)$ is the Lie algebra
generated by the $3r+s$ independent
variables $\{ h_a,e_{i},f_{i}: 1\le a\le r+s,1\le i\le r\}$ subject to the Chevalley-Serre relations:
\beqa
\big[h_a,h_b\big]&=&0\nn\\
\big[e_{i},f_{j}\big]&=&\delta_{ij}\; h_i\nn\\
\big[h_a,e_i\big]&=&\alpha_i(h_a)e_i\ \nn\\
\big[h_a,f_i\big]&=&-\alpha_i(h_a)f_i\ \nn\\
{\rm ad}^{1-A_{ij}}(e_i)\cdot e_j&=&0\nn\\
{\rm ad}^{1-A_{ij}}(f_i)\cdot f_j&=&0\nn
\eeqa
\end{definition}

A last useful definition is the following:
\begin{definition}\label{def:indecomposable}
A matrix is said to be indecomposable if it cannot be made block-diagonal upon simultaneous reordering of its rows and columns.
\end{definition}

\subsection{Spectral graph interpretation and $n$-cubes}\label{subsec:spectral-graph}

\begin{definition}
An undirected loopless multigraph $G$ is defined by a collection of vertices and edges. Let
$\{v_1,\cdots,v_n\}$ be the vertices of $G$ and let $\{E_{ij}: 1\leq i,j\leq n\}$ be the
edges of $G$, {\it i.e.,} $E_{ij}=E_{ji}$ is the number of lines connecting the vertices $v_i$ and $v_j$. Since $G$ is loopless, $E_{ii}=0$ for all $i$. The adjacency matrix of $G$ is $({\rm{Adj}}(G))_{ij}= E_{ij}$.
\end{definition}
\noi
\textit{A priori}, we could have $E_{ii}\neq 0$ in the definition above, however we exclude this possibility in order to associate a Cartan matrix to the graph $G$.\\

To any connected graph $G$ one can associate a symmetric indecomposable Cartan matrix: denoting $G$ the connected graph and {\rm{Adj}}$(G)$ its adjacency matrix, we define
\beqa\label{def:AG}
A_G=2 \mathrm{Id}- {\rm{Adj}}(G)
\eeqa
The matrix $A_G$ is a symmetric indecomposable Cartan matrix in the sense of Definition~\ref{def:KM}.
\\

Thus, in the symmetric graph case, constructing high-corank Cartan matrices is equivalent to constructing connected loopless multigraphs whose adjacency matrix has eigenvalue $2$ with high multiplicity. This observation places the constructions below in the standard language of spectral graph theory; see, for example, \cite{BrouwerHaemers,CvetkovicRowlinsonSimic,GodsilRoyle}.\\

If $G$ is $k$-regular, namely, each vertex has degree $k$, then $\rm{Adj}(G)\mathbf 1=k\mathbf 1$, so $A_G\mathbf 1=(2-k)\mathbf 1$, where $\mathbf 1$ has all entries equal to $1$. Hence a connected regular graph gives an affine-type positive null-vector only when $k=2$. For $k>2$, the Perron--Frobenius eigenvalue of $\rm{Adj}(G)$ is $k>2$, and the zero eigenspace of $A_G$ arises, if at all, from non-Perron eigenspaces of $\rm{Adj}(G)$. Consequently, the kernel vectors necessarily have mixed signs, in agreement with the general Kac--Moody argument of Section~\ref{sec:kernels}.\\

\noi
The Cartan matrices considered in this note are associated to modified $n$-cubes and constructed in an inductive way.
\begin{definition}\label{def:cubes}
An $n$-cube $C_n$ is a graph whose vertices are $\{(p_1\ \cdots \ p_n) \in \{0,1\}^n\}$, and such that any two vertices $(p_1 \ \cdots \ p_n)$ and  $(q_1 \ \cdots \ q_n)$ are connected by one edge if
\beqa
\exists ! i \in\{1,\cdots,n\}:
p_i=q_i+1\ {\rm mod}\ 2 \  \nn
\eeqa
\end{definition}
To each vertex  $(p_1 \ \cdots \ p_n)$ of an $n$-cube is associated a label $I=0,\cdots,2^n-1$ defined as follows:
\beqa
I=\sum \limits_{i=1}^n p_i 2^{i-1}  \nn
\eeqa

\begin{definition}\label{def:simi}
Let $C_{n+1}$ be an $(n+1)$-cube with vertices $(p_1 \ \cdots \ p_{n+1})$. The two $n$-cubes of vertices
 $\{ (p_1 \ \cdots \ p_{n} \ 0) \, , \, p_1,\cdots,p_n \in \{0,1\}\}$ and $\{(p_1 \ \cdots \ p_{n} \ 1), p_1,\cdots,p_n \in \{0,1\}\}$
 are said to be parallel faces of $C_{n+1}$. Furthermore, two vertices $(p_1 \ \cdots \ p_n \ 0)$ and  $(p_1 \ \cdots \ p_n \ 1)$ are said to be similar, and two vertices $(p_1 \ \cdots \ p_{n+1})$ and $(1-p_1 \ \cdots \ 1-p_{n+1})$ ${\rm  mod}$ $2$ are said to be complementary.
\end{definition}

\subsection{A property on the determinant of certain matrices}

To end this section, we recall that if $M$ is the block matrix
\beqa
M=\begin{pmatrix}A&\lambda  \mathrm{Id}\\ \lambda  \mathrm{Id}&D\end{pmatrix}\nn
\eeqa
then
\beqa
\label{eq:det}
\det (M) = \det(AD-\lambda^2\mathrm{Id})
\eeqa
Indeed, we have:
\beqa
\begin{pmatrix}A&\lambda \mathrm{Id}\\ \lambda \mathrm{Id}&D\end{pmatrix}\begin{pmatrix}D&0\\-\lambda \mathrm{Id}&\mathrm{Id}\end{pmatrix}=\begin{pmatrix}AD-\lambda^2\mathrm{Id}&\lambda \mathrm{Id}\\0&D\end{pmatrix}\nn
\eeqa
Then, when $D$ is invertible, we easily get the result by taking the determinant on both sides of this expression. When $D$ is not invertible, we can approach it by a sequence $(D_n)$ of invertible matrices. By continuity of the determinant, we get the result for every matrix.\\

For the recursive constructions below, it is useful to isolate the following elementary spectral lemma, which is often more transparent than the determinant computation alone.

\begin{lemma}[Doubling spectral lemma]\label{lem:doubling-spectrum}
Let $A$ be an $N\times N$ complex matrix and let $L\in\mathbb C$. Set
\[
        \mathcal D_L(A)=
        \begin{pmatrix}
        A&-L \ \mathrm{Id}\\
        -L \ \mathrm{Id}&A
        \end{pmatrix}
\]
If $Ax=\lambda x$, then
\[
        \mathcal D_L(A)\binom{x}{x}=(\lambda-L)\binom{x}{x},
        \qquad
        \mathcal D_L(A)\binom{x}{-x}=(\lambda+L)\binom{x}{-x}
\]
Consequently, with algebraic multiplicities,
\begin{equation}\label{eq:spectrum-doubling}
        \Spec\bigl(\mathcal D_L(A)\bigr)
        =\bigl(\Spec(A)-L\bigr)\sqcup\bigl(\Spec(A)+L\bigr)\nn
\end{equation}
Equivalently, if $P_A(t)=\det(t\mathrm{Id}-A)$, then
\[
        P_{\mathcal D_L(A)}(t)=P_A(t-L)P_A(t+L)
\]
\end{lemma}

\begin{proof}
The subspaces $\{(x,x)^t:x\in\mathbb C^N\}$ and $\{(x,-x)^t:x\in\mathbb C^N\}$ are complementary and invariant under $\mathcal D_L(A)$. On them, the operator acts respectively as $A-L\mathrm{Id}$ and $A+L\mathrm{Id}$. This proves the spectral statement and hence the characteristic-polynomial identity.
\end{proof}

\section{Kernel of Cartan matrices}\label{sec:kernels}

In this section, we recall some basic properties of (generalized) Cartan matrices of corank one, and of corank larger than one by contraposition. The latter properties are applicable to the Cartan matrices of high corank which we build in Section~\ref{sec:constructions}.\\

We first prove a proposition established in \cite{mo, moL}:
\begin{proposition}
\label{prop:moody}
Let $A$ be an $n\times n$ indecomposable Cartan matrix.
\begin{enumerate}[noitemsep]
\item Let $v= \begin{pmatrix}v^1&\cdots&v^n\end{pmatrix}^t$
be a non-zero vector in the kernel of $A$ such that $v^1,\cdots, v^n$ are non-negative, {\it i.e.},
$v^1\ge0,\cdots, v^n\ge 0$, then $v^1,\cdots,v^n>0$.
\item Let $v= \begin{pmatrix}v^1&\cdots&v^n\end{pmatrix}^t$
with $v^1,\cdots,v^n\in \mathbb{R}^+\setminus\{0\}$ and let $w=\begin{pmatrix}w^1&\cdots&w^n\end{pmatrix}^t$ with $w^1,\cdots,w^n\in \mathbb{R}^+\setminus\{0\}$ be vectors of the kernel of A. Then  $w = \lambda v$, $\lambda>0$.
\end{enumerate}
\end{proposition}
\begin{proof}
We reproduce the proof of \cite{mo}.

1.  Let $v=\begin{pmatrix}v^1&\cdots&v^n\end{pmatrix}^t$ be a null-vector of $A$ such that $v^i\geq 0$, for all $i$. Since $A_{ii}=2$ and $A_{ij}\leq 0 \text{ for }i\ne j$,
\beqa
 Av=0 \ \  \Longleftrightarrow\ \
\forall i \in \{1,\cdots,n\} \ , \ \sum \limits_{j\ne i} \big|A_{ij}\big| v^j =2 v^i\nn\
\eeqa
Write $\{1,\cdots,n\}=I\cup J$ where $I=\{i: v^i=0\}$ and $J=\{i:v^i>0\}$.
The relation above yields
\beqa
 \ \forall i \in I\ ,\ \sum \limits_{j\in J} \big|A_{ij}\big|\;v^j=0\ \nn
\eeqa
since $\big|A_{ij}\big|v^j=0$, for $j\in I$. Since $v^j>0$ by definition of $J$, $A_{ij}=0$, for all $i \in I, j\in J$.
By hypothesis, $A$ is indecomposable, thus either $I=\emptyset$ or $J=\emptyset$. Since
$\dim (\text{Ker}( A))\ge 1$, $I=\emptyset$ and $v^1,\cdots,v^n>0$. This proves Proposition 1.

2. Let $v=\begin{pmatrix}v^1&\cdots&v^n\end{pmatrix}$, $v^i>0$, and $w=\begin{pmatrix}w^1&\cdots&w^n\end{pmatrix}$, $w^i>0$, be vectors in the kernel of $A$. Defining $\lambda=\min\left(w^i/v^i\right) \in \mathbb{R}^+\setminus\{0\}$, one finds that the vector $w-\lambda v$ is in the kernel, and has all its entries non-negative and at least one zero. By 1., this is not possible unless $w-\lambda v=0$.
\end{proof}

Thus if $A$ is an indecomposable Cartan matrix such that the equation
$Av=0$  is {\it only } satisfied by vectors $v$ with strictly positive components, Cork$(A)=1$. Moody and collaborators further proved that $A$ is the Cartan matrix of an affine Lie algebra \cite{mo}.\\

In the Proposition above, we did not consider the possibility that the kernel of $A$ contains a vector $v$ with positive components and another $w$ with positive and negative components. This case is actually excluded, and the existence of $v$ forbids the existence of $w$.
\\
\begin{corollary}\label{cor:inde}
Let $A$ be an indecomposable Cartan matrix such that Ker$(A)$ contains a vector $v$ with strictly positive entries. Then Cork$(A)=1$.
\end{corollary}
\begin{proof}
Let $v=\begin{pmatrix}v^1&\cdots&v^n\end{pmatrix}^t$, $v^1,\cdots, v^n \in \mathbb{R}^+\setminus\{0\}$ be such that $Av=0$. Consider then $w\in \text{Ker}(A)$.
We can find a positive $C$ such that all of the entries of the vector $w+C v$ are strictly positive. Thus by point 2. of Proposition~\ref{prop:moody} there exists $\lambda\in \mathbb{R}^+\setminus\{0\}$ such that
\beqa
w+Cv =\lambda v \ \ ,\textit{ i.e.}, \ \
w= (\lambda-C) v \ \nn
\eeqa
$w$ and $v$ are thus linearly dependent, therefore Ker$(A)=$Vec$(v)$.
\end{proof}

We have seen that if an indecomposable Cartan matrix $A$ admits one vector with strictly positive entries then Cork$(A)=1$. One may wonder if the converse applies, {\it i.e.},  whether the kernel of a Cartan matrix of corank $1$ only contains vectors of strictly positive or strictly negative entries. It turns out not to be the case, as the following example demonstrates.
\begin{example}
Consider the indecomposable $6\times 6$ Cartan matrix $$
A=\left(\begin{array}{rrrrrr}
2 & -2 & -2 & 0 & -2 & 0
\\
 -2 & 2 & 0 & -2 & 0 & -2
\\
 -2 & 0 & 2 & -2 & 0 & 0
\\
 0 & -2 & -2 & 2 & 0 & 0
\\
 -2 & 0 & 0 & 0 & 2 & -2
\\
 0 & -2 & 0 & 0 & -2 & 2
\end{array}\right)
$$
It is of corank $1$, and the vector
$$ v=\begin{pmatrix}
\phantom{-}0& \phantom{-}0& -1& -1& \phantom{-}1& \phantom{-}1
\end{pmatrix}^t
$$
is in the kernel of $A$, despite having positive, negative and zero entries.
\end{example}

The previous corollary has the following immediate but important consequence. Let $A$ be indecomposable and suppose that $\Cork(A)>1$. Then, no non-zero element of $\ker (A)$ can be entrywise non-negative or entrywise non-positive. Indeed, if $0\neq v\in\ker(A)$ were entrywise non-negative, Proposition~\ref{prop:moody} would imply that all its entries are strictly positive, and Corollary~\ref{cor:inde} would force $\Cork(A)=1$, that is a contradiction. Applying the same argument to $-v$ excludes the entrywise non-positive case. Thus every non-zero vector in $\ker(A)$ has at least one positive and at least one negative entry.

Since $A$ has integer entries, $\ker(A)$ admits a basis over $\mathbb Q$ and hence an integral basis after clearing denominators. Let $D=\diag(d_1,\ldots,d_n)$, with $d_i>0$, be such that $DA$ is symmetric, and let $\alpha_1,\ldots,\alpha_n$ be the simple roots of the Kac--Moody algebra associated with $A$. On the span of the simple roots, we use the standard symmetric bilinear form
\[
        (\alpha_i,\alpha_j)=d_i A_{ij}
\]
For $v=(v^1,\ldots,v^n)^t\in\ker(A)$, define
\[
        \delta_v=\sum_{i=1}^n v^i\alpha_i
\]
Then, for every simple root $\alpha_j$,
\begin{equation}\label{eq:delta-orth-simple}
        (\delta_v,\alpha_j)
        =\sum_{i=1}^n v^i d_i A_{ij}
        =d_j\sum_{i=1}^n A_{ji}v^i=0\nn
\end{equation}
where we used $d_iA_{ij}=d_jA_{ji}$. Hence $\delta_v$ belongs to the radical of the bilinear form on the root lattice. In particular,
\[
        (\delta_v,\delta_v)=0
\]
If $v$ is integral, then $\delta_v$ belongs to the root lattice $Q=\bigoplus_i\mathbb Z\alpha_i$.\\

One must be slightly careful about terminology. A root of a Kac--Moody algebra is either positive or negative with respect to the chosen simple system; hence its coordinates in the basis of simple roots are either all non-negative or all non-positive. Therefore a lattice vector $\delta_v$ with coefficients of mixed sign cannot be a root. The converse is not asserted: having coefficients all of the same sign is a necessary condition for a root, not a sufficient one.

\begin{corollary}\label{cor:isotropic-not-root}
Let $A$ be an indecomposable symmetrizable generalized Cartan matrix with $\Cork(A)>1$. For every non-zero integral vector $v\in\ker(A)$, the vector $\delta_v=\sum_i v^i\alpha_i$ is an isotropic element of the root lattice and is orthogonal to the whole root lattice. Moreover, $\delta_v$ is not a root of the Kac--Moody algebra associated with $A$.
\end{corollary}

\section{Explicit constructions}\label{sec:constructions}

In this section, we systematically construct Kac--Moody algebras of high corank from the simple Lie algebra $\mathfrak{a}_3$ or $\mathfrak{a}_1$. The general strategy, detailed thereafter, is to construct Cartan matrices by induction, namely the order $(n+1)$ Cartan matrix $A_{n+1}$ from the order $n$ Cartan matrix $A_n$ as follows:
\beqa
A_{n+1}=\begin{pmatrix} A_n&-L \ \mathrm{Id}\\ -L \ \mathrm{Id}&A_n\end{pmatrix}\nn
\eeqa
where $\operatorname{Id}$ is the identity matrix and $L$ a positive integer. As we explain below, this has the geometrical interpretation of gluing together certain $n$-cubes, where $L$ represents the number of links between two similar vertices.\\

In this paper we focus on two families of constructions involving large matrices, and for each, we perform the analysis for $L=1,2$. For the first family, the three first steps of the procedure are special cases, while the next ones are perfectly recursive, allowing to obtain rapidly growing coranks. The second family generalises the construction and is systematic, {\it i.e.} all steps are the same. Different behaviors appear for different values of $L$. All the eigenvalues of the constructed matrices are integers, and a large number of them are negative. Finally, the corresponding Dynkin diagrams are regular, {\it i.e.}, every vertex is connected to the same number of edges.

Throughout the spectral formulae below, binomial coefficients are understood with the convention
\[
        \binom{N}{j}=0\qquad \text{if } j\notin\mathbb Z \text{ or } j<0 \text{ or } j>N
\]
which allows the formulae to be written without repeatedly listing divisibility and parity conditions.

\subsection{First family of constructions}
\label{pyramid}
The first family is obtained by starting with certain $n$-cubes $C_n$ of which we remove some vertices. As said before, the first three steps of the construction are special, and are followed by a recursive procedure.\\

{\bf Step one}\\
Let $C_2$ be the two-cube, {\it i.e.} the square with vertices
\beqa
0=(0 \, 0) \ , \qquad  1=(1 \, 0) \ , \qquad 2=(0 \, 1) \ , \qquad 3=(1 \, 1) \nn
\eeqa
(and edges linking them, as per Definition~\ref{def:cubes}). Consider now the graph obtained from $C_2$ by removing the $0$-th vertex and the edges connected to it. To graphically represent this graph, as well as any other obtained from $n$-cubes by a similar procedure henceforth, we abide by the following conventions: vertices which are kept are represented by dots, removed edges are drawn black while those that are kept are coloured. This yields
\begin{center}
  \begin{tikzpicture}[thick,scale=.6]
  \coordinate (P0) at (0,0);
    \coordinate (P1) at (2,0) ; 
    \coordinate (P2) at (0,2); 
    \coordinate (P3) at (2,2);

 \draw[line width=1pt] (P0) -- (P1);
 \draw[line width=1pt,color=red] (P1) -- (P3);
 \draw[line width=1pt,color=red] (P2) -- (P3);
 \draw[line width=1pt] (P2) -- (P0);

\draw (P1)  node {$\bullet$} ;
\draw (P2)  node{$\bullet$} ;
\draw (P3)  node{$\bullet$} ;

\draw (P0) node[left]{$0$} ;
\draw (P1) node[right]{$1$} ;
    \draw (P2) node[left]{$2$} ;
    \draw (P3) node[ right]{$3$} ;
   \end{tikzpicture}
\end{center}
\noi To the graph above, reordering the labels $(1,3,2)\to (1,2,3)$, we associate the following Cartan matrix and the corresponding Dynkin diagram (see \eqref{def:AG}):
\beqa
\hskip 3.truecm A_1=\left(\begin{array}{rrr}2&-1&0\\-1&2&-1\\0&-1&2\end{array}\right)
\hskip 2.truecm \begin{minipage}{10cm}
\begin{tikzpicture}[scale=.4]

    \tikzstyle{point}=[circle,draw, thick, minimum size=.5cm] ] (6. cm)  ;
      \tikzstyle{ligne}=[thick]
      \tikzstyle{pointille}=[thick,dotted]
      \node (1) at ( 0,0) [point]  {};
      \draw (1)  node {$1$};
      \node (2) at ( 4,0)  [point]  {};
      \draw (2)  node {$2$};
      \node (3) at ( 8,0)  [point]  {};
      \draw (3)  node {$3$};

      \draw[line width=1pt,color=red]  (1) -- (2);
      \draw[line width=1pt,color=red]  (2) -- (3);
        \end{tikzpicture}
        \end{minipage}
\nn
\eeqa
  which is the Dynkin diagram of the simple Lie algebra $\mathfrak{a}_3$. Obviously Cork$(A_1)=0$.\\

{\bf Step two}\\
Consider now two parallel $2$-cubes $C_2^0$ with vertices $\{0=(0 \, 0 \, 0),1=(1 \, 0 \, 0),2=(0 \, 1 \, 0),3=(1 \, 1 \, 0)\}$ and $C_2^1$ with vertices $\{4=(0 \, 0 \, 1),5=(1 \, 0 \, 1),6=(0 \, 1 \, 1),7=(1 \, 1 \, 1)\}$. We start by removing the $0=(0 \, 0 \, 0)$ vertex from the $C_2^0$ cube and the complementary $7=(1 \, 1 \, 1)$ vertex from the $C_2^1$ cube. Then, the two $2$-cubes are connected by linking similar vertices (see Definition \ref{def:simi}), which yields:

\begin{center}
  \begin{tikzpicture}[thick,scale=.6]
  \coordinate (P0) at (0,0);
    \coordinate (P1) at (4.5,0) ; 
    \coordinate (P4) at (6.6,2.1); 
    \coordinate (P2) at (2.1,2.1);

     \coordinate (P3) at (0,3);
    \coordinate (P7) at  (4.5,3) ;
    \coordinate (P6) at (6.6,5.1); 
    \coordinate (P5) at (2.1,5.1);

  \draw[solid][line width=1pt] (P7)-- (P6) ;
\draw[solid][line width=1pt] (P6)-- (P5) ;
\draw[solid][line width=1pt,color=red] (P3)-- (P5) ;
\draw[solid][line width=1pt,color=red] (P3)-- (P7) ;

 \draw[solid][line width=1pt] (P0) -- (P1);
 \draw[dashed][line width=1pt] (P0) -- (P2);
   \draw[solid][line width=1pt,color=red] (P1) -- (P4);
 \draw[dashed][line width=1pt,color=red] (P2) -- (P4);

 \draw[solid][line width=1pt] (P0) -- (P3);
 \draw[solid,color=blue][line width=1pt] (P1) -- (P7);
\draw[solid][line width=1pt] (P4) -- (P6);
\draw[dashed][line width=1pt,color=blue] (P2) -- (P5);

\draw (P1)  node {$\bullet$} ;
\draw (P2)  node{$\bullet$} ;
\draw (P3)  node{$\bullet$} ;
\draw (P4)  node{$\bullet$} ;
\draw (P5)  node{$\bullet$} ;
\draw (P7)  node{$\bullet$} ;

\draw (P0) node[below]{$0$} ;
\draw (P1) node[below ]{$1$} ;
    \draw (P2) node[above right]{$2$} ;
    \draw (P3) node[left]{$4$} ;
    \draw (P4) node[above right]{$3$} ;
    \draw (P5) node[ left]{$6$} ;
    \draw (P6) node[right]{$7$} ;
    \draw (P7) node[right=1pt]{$5$} ;
   \end{tikzpicture}
\end{center}

\noi
The two subgraphs $C_2^0$ and $C_2^1$ are identical to that of the previous step, which we highlight by drawing their edges in red. Blue edges are additional ones corresponding to the gluing, while black ones are removed (in agreement with our conventions). Note that the graph produced is $2$-regular, \textit{i.e.} each vertex is connected to two edges. Reordering the labels as follows: $(1,3,2,6,4,5)\to(1,2,3,4,5,6)$, we associate to the graph above the Cartan matrix and Dynkin diagram:

\beqa
\hskip 1.truecm A_2=\begin{pmatrix}2&-1&0&0&0&-1\\
                   -1&2&-1&0&0&0\\
                    0&-1&2&-1&0&0\\
                    0&0&-1&2&-1&0\\
                    0&0&0&-1&2&-1\\
                    -1&0&0&0&-1&2
                    \end{pmatrix}
\hskip 1.truecm \begin{minipage}{10cm}
\begin{tikzpicture}[scale=.35]

    \tikzstyle{point}=[circle,draw, thick, minimum size=.5cm] ] (6. cm)  ;
      \tikzstyle{ligne}=[thick]
      \tikzstyle{pointille}=[thick,dotted]
      \node (1) at ( 0,0) [point]  {};
      \draw (1)  node {$1$};
      \node (2) at ( 4,0)  [point]  {};
      \draw (2)  node {$2$};
      \node (3) at ( 8,0)  [point]  {};
      \draw (3)  node {$3$};
      \node (4) at ( 12,0)  [point]  {};
      \draw (4)  node {$4$};
      \node (5) at ( 16,0)  [point]  {};
      \draw (5)  node {$5$};
       \node (6) at ( 8,4)  [point]  {};
      \draw (6)  node {$6$};

      \draw[line width=1pt,color=red]  (1) -- (2);
      \draw[line width=1pt,color=red]  (2) -- (3);
      \draw[line width=1pt,color=blue]  (3) -- (4);
      \draw[line width=1pt,color=red]  (4) -- (5);
      \draw[line width=1pt,color=red]  (5) -- (6);
      \draw[line width=1pt,color=blue]  (6) -- (1);
        \end{tikzpicture}
        \end{minipage}
\nn
                   \eeqa

\noi
corresponding to the affine Lie algebra $\widehat{\mathfrak{a}}_5$. We find easily that Cork$(A_2)=1$ and that the vector $v_3=\begin{pmatrix}1&1&1&1&1&1\end{pmatrix}^t$ generates the kernel of $A_2$.\\

{\bf Step three}\\
Finally, consider two parallel $3$-cubes: $C_3^0$ with vertices $\{0=(0 \, 0 \, 0 \, 0), 1=(1 \, 0 \, 0 \, 0), 2=(0 \, 1 \, 0 \, 0), 3=(1 \, 1 \, 0 \, 0), 4=(0 \, 0 \, 1 \, 0),5=(1 \, 0 \, 1 \, 0), 6=(0 \, 1 \, 1 \, 0), 7=(1 \, 1 \, 1 \, 0)\}$ and $C_3^1$ with vertices $\{8=(0 \, 0 \, 0 \, 1), 9=(1 \, 0 \, 0 \, 1), 10=(0 \, 1 \, 0 \, 1), 11=(1 \, 1 \, 0 \, 1),12=(0 \, 0 \, 1 \, 1), 13=(1 \, 0 \, 1 \, 1), 14=(0 \, 1 \, 1 \, 1), 15=(1 \, 1 \, 1 \, 1)\}$. We first remove the vertices $0=(0 \, 0 \, 0 \, 0)$ and $7=(1 \, 1 \, 1 \, 0)$ from the $C_3^0$ cube and the complementary vertices $15=(1 \, 1 \, 1 \, 1)$ and $8=(0 \, 0 \, 0 \, 1)$ from the parallel $C_3^1$ cube. Then, we connect the two $3$-cubes by linking similar vertices:

\begin{center}
  \begin{tikzpicture}[thick,scale=.6]
  \coordinate (P0) at (0,0);
    \coordinate (P1) at (4.5,0) ; 
    \coordinate (P4) at (6.6,2.1); 
    \coordinate (P2) at (2.1,2.1);

     \coordinate (P3) at (0,3);
    \coordinate (P7) at  (4.5,3) ;
    \coordinate (P6) at (6.6,5.1); 
    \coordinate (P5) at (2.1,5.1);

  \draw[solid][line width=1pt] (P7)-- (P6) ;
\draw[solid][line width=1pt] (P6)-- (P5) ;
\draw[solid][line width=1pt,color=red] (P3)-- (P5) ;
\draw[solid][line width=1pt,color=red] (P3)-- (P7) ;

 \draw[solid][line width=1pt] (P0) -- (P1);
 \draw[dashed][line width=1pt] (P0) -- (P2);
   \draw[solid][line width=1pt,color=red] (P1) -- (P4);
 \draw[dashed][line width=1pt,color=red] (P2) -- (P4);

 \draw[solid][line width=1pt] (P0) -- (P3);
 \draw[solid,color=red][line width=1pt] (P1) -- (P7);
\draw[solid][line width=1pt] (P4) -- (P6);
\draw[dashed][line width=1pt,color=red] (P2) -- (P5);

\draw (P1)  node {$\bullet$} ;
\draw (P2)  node{$\bullet$} ;
\draw (P3)  node{$\bullet$} ;
\draw (P4)  node{$\bullet$} ;
\draw (P5)  node{$\bullet$} ;
\draw (P7)  node{$\bullet$} ;
\draw (P0) node[below]{$0$} ;
\draw (P1) node[below ]{$1$} ;
    \draw (P2) node[above right]{$2$} ;
    \draw (P3) node[left]{$4$} ;
    \draw (P4) node[above right]{$3$} ;
    \draw (P5) node[ left]{$6$} ;
    \draw (P6) node[right]{$7$} ;
    \draw (P7) node[right=1pt]{$5$} ;

\coordinate (PP0) at (0,6);
    \coordinate (PP1) at (4.5,6) ; 
    \coordinate (PP4) at (6.6,8.1); 
    \coordinate (PP2) at (2.1,8.1);

     \coordinate (PP3) at (0,9);
    \coordinate (PP7) at  (4.5,9) ;
    \coordinate (PP6) at (6.6,11.1); 
    \coordinate (PP5) at (2.1,11.1);

\draw (PP1)  node {$\bullet$} ;
\draw (PP2)  node{$\bullet$} ;
\draw (PP3)  node{$\bullet$} ;
\draw (PP4)  node{$\bullet$} ;
\draw (PP5)  node{$\bullet$} ;
\draw (PP7)  node{$\bullet$} ;

  \draw[solid][line width=1pt] (PP7)-- (PP6) ;
\draw[solid][line width=1pt] (PP6)-- (PP5) ;
\draw[solid][line width=1pt,color=red] (PP3)-- (PP5) ;
\draw[solid][line width=1pt,color=red] (PP3)-- (PP7) ;

 \draw[solid][line width=1pt] (PP0) -- (PP1);
 \draw[dashed][line width=1pt] (PP0) -- (PP2);
   \draw[solid][line width=1pt,color=red] (PP1) -- (PP4);
 \draw[dashed][line width=1pt,color=red] (PP2) -- (PP4);

 \draw[solid][line width=1pt] (PP0) -- (PP3);
 \draw[solid,color=red][line width=1pt] (PP1) -- (PP7);
\draw[solid][line width=1pt] (PP4) -- (PP6);
\draw[dashed][line width=1pt,color=red] (PP2) -- (PP5);

\draw (PP0) node[below]{$8$} ;
\draw (PP1) node[below ]{$9$} ;
    \draw (PP2) node[above right]{$10$} ;
    \draw (PP3) node[left]{$12$} ;
    \draw (PP4) node[above right]{$11$} ;
    \draw (PP5) node[ left]{$14$} ;
    \draw (PP6) node[right]{$15$} ;
    \draw (PP7) node[right=1pt]{$13$} ;

\draw (PP1)[color=blue] to[out=0,in=0, looseness=1.] (P1);
\draw (PP2)[color=blue] to[out=180,in=180, looseness=.5] (P2);
\draw (PP3)[color=olive] to[out=180,in=180, looseness=.5] (P3);
\draw (PP4)[color=blue] to[out=0,in=0, looseness=.5] (P4);
\draw (PP5)[color=olive] to[out=180,in=180, looseness=.3] (P5);
\draw (PP7)[color=olive] to[out=0,in=0, looseness=.5] (P7);
   \end{tikzpicture}
\end{center}
\noi The two subgraphs $C_3^0$ and $C_3^1$ are the same as in the previous step, which we highlight by drawing the corresponding  edges  in red.  Blue and green edges arise from the gluing. The graph produced is $3$-regular. Reordering the vertices $(1,3,2,6,4,5)\to(1,2,3,4,5,6)$ and $(9,11,10,14,12,13)\to(7,8,9,10,11,12)$, one obtains the Cartan matrix
\beqa
A_3=\begin{pmatrix}A_2&-\mathrm{Id}\\-\mathrm{Id}&A_2\end{pmatrix}\nn\
\eeqa
where $A_2$ is the matrix of step two and Id the $6\times 6$ identity matrix. We dub the corresponding Dynkin diagram ``the aztec pyramid":

\begin{center}
\begin{tikzpicture}[scale=.3]

   \tikzstyle{point}=[circle,draw, thick, minimum size=.5cm] ] (6. cm)  ;
\tikzstyle{ligne}=[thick]
\tikzstyle{pointille}=[thick,dotted]
\node (0) at ( 0,0) [point]  {};
\draw (0)  node {$1$};

\node (2) at ( 0,4)  [point]  {};
\draw (2)  node {$2$};
\node (4) at ( -3,7)  [point]  {};
\draw (4)  node {$3$};
\node (10) at ( -6,10)  [point]  {};
\draw (10)  node {$4$};
\node (8) at ( -6,-6)  [point]  {};
\draw (8)  node {$5$};
\node (6) at ( -3,-3)  [point]  {};
\draw (6)  node {$6$};
\draw[line width=1pt,color=red]  (0) -- (2);
\draw[line width=1pt,color=red]  (2) -- (4);
\draw[line width=1pt,color=red]  (4) -- (10);
\draw[line width=1pt,color=red]  (10) -- (8);
\draw[line width=1pt,color=red]  (8) -- (6);
\draw[line width=1pt,color=red]  (6) -- (0);
\node (1) at ( 4,0)  [point]  {};
\draw (1)  node {$7$};
\node (3) at ( 4,4)  [point]  {};
\draw (3)  node {$8$};
\node (5) at ( 7,7) [point]  {};
\draw (5)  node {$9$};
\node (11) at ( 10,10) [point]  {};
\draw (11)  node {$10$};
\node (9) at ( 10,-6)  [point]  {};
\draw (9)  node {$11$};
\node (7) at ( 7,-3)  [point]  {};
\draw (7)  node {$12$};
\draw[line width=1pt,color=red]  (1) -- (3);
\draw[line width=1pt,color=red]  (3) -- (5);
\draw[line width=1pt,color=red]  (5) -- (11);
\draw[line width=1pt,color=red]  (11) -- (9);
\draw[line width=1pt,color=red]  (9) -- (7);
\draw[line width=1pt,color=red]  (7) -- (1);
\draw[line width=1pt,color=blue]  (0) -- (1);
\draw[line width=1pt,color=blue]  (2) -- (3);
\draw[line width=1pt,color=blue]  (4) -- (5);
\draw[line width=1pt,color=olive]  (10) -- (11);
\draw[line width=1pt,color=olive]  (8) -- (9);
\draw[line width=1pt,color=olive]  (6) -- (7);

\end{tikzpicture}
\end{center}

\noi
The Cartan matrix is such that Cork$(A_3)=2$, and the study of the associated Lie algebra is performed in Appendix \ref{app:A4}. Having found a method well-suited to obtain Kac--Moody algebras of corank higher than $1$, we proceed to the general step. Beyond this point, we do not fix the number of links $L$ between vertices.\\

{\bf Higher-order steps}\\
Starting from the $4$-cube with missing vertices of the previous step, we now simply iterate the following procedure: we build an $(n+1)$-cube by duplicating the $n$-cube and connecting similar vertices with $L$ links. At step $n$, we thus get an $(n+1)$-cube with some missing vertices and some edges made of $L$ links. For simplicity, we will call them cubes but keep in mind this subtlety. \\

The associated Cartan matrices also satisfy a recursive pattern: from the Cartan matrix $A^{(L)}_{n}$ of the $(n+1)$-cube of step $n$, we obtain the Cartan matrix at step $n+1$ (associated to an $(n+2)$-cube) as:
\beqa
A^{(L)}_{n+1} =\begin{pmatrix} A^{(L)}_{n}&-L \ \mathrm{Id}\\-L \ \mathrm{Id}&A^{(L)}_{n} \end{pmatrix}\ \nn
\eeqa
where $\operatorname{Id}$ is the identity matrix. The block form of the matrix $A^{(L)}_{n+1}$  allows us to claim, using \eqref{eq:det}, that
\beqa
\det (\lambda \mathrm{Id}-A^{(L)}_{n+1})=
\det\big((\lambda-L) \mathrm{Id}-A^{(L)}_{n}\big)\det\big((\lambda+L) \mathrm{Id}-A^{(L)}_{n})\big)\nn
\eeqa
Denoting $P^{(L)}_n$ the characteristic polynomial of $A^{(L)}_{n}$, we thus have the recursive
relation
\beqa
\label{eq:Pnn+1}
P^{(L)}_{n+1}(\lambda)= P^{(L)}_n(\lambda-L)\times P^{(L)}_n(\lambda+L)
\eeqa
Let $\lambda$ be an eigenvalue of $P^{(L)}_{n}$ and let $d^{(L)}_n(\lambda)$ be the
degeneracy of the eigenvalue $\lambda$ for $A^{(L)}_{n}$. From \eqref{eq:Pnn+1} we find:
\beqa
\label{eq:dLn+1}
d^{(L)}_{n+1}(\lambda)=d^{(L)}_n(\lambda-L)+d^{(L)}_n(\lambda+L)
\eeqa
The recursive relation determines the full spectrum for all $n\geq3$ and all positive integers $L$. The seed spectrum is
\[
        P_3(\lambda)=\lambda^{2}(\lambda+1)(\lambda-1)(\lambda-2)^4(\lambda-3)(\lambda-4)^2(\lambda-5)
\]
so the seed eigenvalues and multiplicities are
\begin{equation}\label{eq:first-seed-spectrum}
\mathcal S_3=\{-1,0,1,2,3,4,5\},
\quad
d_3(-1),d_3(0),d_3(1),d_3(2),d_3(3),d_3(4),d_3(5)
=(1,2,1,4,1,2,1)\nn
\end{equation}

\begin{theorem}[Spectrum of the first family for any $L$]\label{thm:first-family-all-L}
For $n\geq3$, the spectrum of $A_n^{(L)}$, counted with algebraic multiplicities, is given by
\begin{equation}\label{eq:first-spectrum-all-L}
        \Spec\bigl(A_n^{(L)}\bigr)
        =
        \left\{
        \lambda_{\mathcal S_3}+L(n-3-2j):
        \lambda_{\mathcal S_3}\in\mathcal S_3,
        \;0\leq j\leq n-3
        \right\}
\end{equation}
where the contribution of the pair $(\lambda_{\mathcal S_3},j)$ has multiplicity
\begin{equation}\label{eq:first-multiplicity-all-L}
        d_3(\lambda_{\mathcal S_3})\binom{n-3}{j}
\end{equation}
Equivalently, for any $\lambda\in\mathbb Z$,
\begin{equation}\label{eq:first-degeneracy-all-L}
        d^{(L)}_n(\lambda)
        =
        \sum_{\lambda_{\mathcal S_3}\in\mathcal S_3}
        d_3(\lambda_{\mathcal S_3})
        \binom{n-3}{\frac{n-3}{2}+\frac{\lambda_{\mathcal S_3}-\lambda}{2L}}
\end{equation}
In particular,
\begin{equation}\label{eq:first-corank-all-L}
        \Cork\bigl(A_n^{(L)}\bigr)
        =
        \sum_{\lambda_{\mathcal S_3}\in\mathcal S_3}
        d_3(\lambda_{\mathcal S_3})
        \binom{n-3}{\frac{n-3}{2}+\frac{\lambda_{\mathcal S_3}}{2L}}
\end{equation}
\end{theorem}

\begin{proof}
The assertion follows by iterating Lemma~\ref{lem:doubling-spectrum}. At each of the $n-3$ recursive doublings, an eigenvalue is shifted either by $+L$ or by $-L$. If exactly $j$ of these shifts are $-L$ and $n-3-j$ are $+L$, the total shift is $L(n-3-2j)$, and there are $\binom{n-3}{j}$ such choices. Multiplying by the seed multiplicity $d_3(\lambda_{\mathcal S_3})$ gives \eqref{eq:first-multiplicity-all-L}. Formulae \eqref{eq:first-degeneracy-all-L} and \eqref{eq:first-corank-all-L} are obtained by solving $\lambda=\lambda_{\mathcal S_3}+L(n-3-2j)$ for $j$.
\end{proof}

We emphasize the following: the eigenvalues are not, for arbitrary $L$, all consecutive integers between the two extreme values. They are obtained from the seed eigenvalues by the shifts in \eqref{eq:first-spectrum-all-L}. For instance, at step $4$ and $L=4$ one obtains
\[
\{-5,-4,-3,-2,-1,0,1\}\sqcup\{3,4,5,6,7,8,9\}
\]
so the integer $2$ is absent. The eigenvalues are consecutive only when the shifted seed intervals overlap sufficiently, for example for $L=1,2,3$ in the first step after the seed.

In order to gain insight on our construction and supplement the results of Theorem~\ref{thm:first-family-all-L}, we derive the same results from another perspective. In this approach, the degeneracy formulas will be more compact and thus easier to handle. We iterate once \eqref{eq:dLn+1} to get
\beqa
d^{(L)}_{n+2}(\lambda)=d^{(L)}_n(\lambda-2L)+2 d^{(L)}_n(\lambda)+d^{(L)}_n(\lambda+2L)  \nn
\eeqa
This recursive relation, involving only eigenvalues equal modulo $2L$, is analogous to another involving binomial coefficients:
\beqa
\begin{pmatrix}i+2\\j+1\end{pmatrix}=
\begin{pmatrix}i \\j+1\end{pmatrix}+
2 \begin{pmatrix}i \\j\end{pmatrix}+
\begin{pmatrix}i \\j-1\end{pmatrix} \nn
\eeqa
The recursion relation being linear, we thus look for $d^{(L)}_n(\lambda)$ as sums of coefficients of the form
\beqa\label{eq:degeneracyAnsatz}
a \begin{pmatrix}n-b\\ \frac{\lambda-r}{2L} +\frac{n-k}2 +c\end{pmatrix}
\eeqa
where the $n,\lambda$ dependence is dictated by the identification of the two analogous recursive relations: the top integer should increase by $2$ when $n$ does, while the bottom integer should increase by $1$ when either $\lambda$ increases by $2L$ or $n$ increases by $2$. The integer coefficients $a,b,c,k,r$ are identified by matching some initial conditions, with the constraints that $k = n$ mod $2$ and $r = \lambda $ mod $2L$. (Actually, $c=0$ or $1$ is sufficient as other values can be reabsorbed in $k$, while values other than $k=0,...,2L-1$ can be reabsorbed in $r$.) Since the recursion relation relates degeneracies at step $n$ and $n+2$, odd and even values of $n$ have to be distinguished.\\

\subsubsection{$L=2$ case}

We now focus on the case where double links connect the cubes, \textit{i.e.} $L=2$. According to the previous discussion, there exists a recursive relation between the degeneracies of eigenvalues equal modulo $4$ at steps $n$ and $n+2$. In order to identify the coefficients $a,b,c,k,r$ of the general ansatz for the degeneracy \eqref{eq:degeneracyAnsatz}, we consider the two characteristic polynomials $P_3$ and $P^{(2)}_4$:
\beqa
P_3(\lambda)=\lambda^{2} \left(\lambda +1\right)\left(\lambda -1\right)
\left(\lambda -2\right)^{4} \left(\lambda -3\right)\left(\lambda -4\right)^{2}
\left(\lambda -5\right) \nn
\eeqa
\beqa
P^{(2)}_4(\lambda)= \lambda^4 (\lambda +3) (\lambda +2)^2 (\lambda +1) (\lambda -1)^2 (\lambda -2)^4  (\lambda -3)^2 (\lambda -4)^4  (\lambda -5) (\lambda -6)^2  (\lambda -7) \nn
\eeqa
Let us detail the computation for the case of $d^{(2)}_{n=3}(\lambda =0\mod 4)$. We notice that
\beqa
d^{(2)}_3(0)=2=2\begin{pmatrix}1\\0\end{pmatrix} , \qquad d^{(2)}_3(4)=2\begin{pmatrix}1\\1\end{pmatrix}\nn
\eeqa
while, assuming that a single term of the form \eqref{eq:degeneracyAnsatz} is sufficient,
\beqa
d^{(2)}_3(0)=a\begin{pmatrix}3-b\\\frac{-r}{4}+\frac{3-k}{2}+c\end{pmatrix} , \qquad d^{(2)}_3(4)=a\begin{pmatrix}3-b\\\frac{4-r}{4}+\frac{3-k}{2}+c\end{pmatrix}\nn
\eeqa
One easily sees that $a=2,b=2,c=0,k=3,r=0$ is a solution. Following the same reasoning, we obtain all degeneracies reported in Table~\ref{table:degFirstL2}.

\begin{table}[h!]
\centering
\begin{tabular}{c|cccc}
$d^{(2)}_n(\lambda)$ & $\ \ \lambda=0\ \  \text{mod}\ \  4 \ \ $ & $\ \  \lambda=1\ \  \text{mod}\ \  4 \ \  $ & $\ \  \lambda=2\ \  \text{mod}\ \  4\ \  $ & $\ \  \lambda=3\ \  \text{mod}\ \  4\ \  $ \\
\hline
\vspace{-8pt}& & & &\\
even $n$ & $4 \begin{pmatrix}n-3\\\frac \lambda 4 +\frac{n-4}2\end{pmatrix}$ & $\begin{pmatrix}n-2\\\frac {\lambda+3} 4 +\frac{n-4}2\end{pmatrix}$ & $2 \begin{pmatrix}n-2\\\frac {\lambda+2} 4 +\frac{n-4}2\end{pmatrix}$ & $\begin{pmatrix}n-2\\\frac {\lambda+1} 4 +\frac{n-4}2\end{pmatrix}$ \\
\vspace{-8pt}& & & &\\
\hline
\vspace{-8pt}& & & &\\
odd $n$ & $2 \begin{pmatrix}n-2\\\frac \lambda 4 +\frac{n-3}2\end{pmatrix} $ & $\begin{pmatrix}n-2\\\frac {\lambda+3} 4 +\frac{n-5}2\end{pmatrix}$ & $4 \begin{pmatrix}n-3\\\frac {\lambda-2} 4 +\frac{n-3}2\end{pmatrix} $ & $\begin{pmatrix}n-2\\\frac {\lambda+1} 4 +\frac{n-3}2\end{pmatrix}$\\
\end{tabular}
\caption{Values of the degeneracies $d^{(2)}_n(\lambda)$ for the first family of constructions when $L=2$.}
\label{table:degFirstL2}
\end{table}
In particular, we derive a general formula for the corank:
\beqa
\text{Cork}(A^{(2)}_{n})= \left\{
\begin{array}{cll}
4 \begin{pmatrix}n-3\\\frac{n-4}2\end{pmatrix}= 2 \begin{pmatrix}n-2\\\frac{n-2}2\end{pmatrix}&\text{if}&n\ \ \text{is even}\ \\
 2 \begin{pmatrix}n-2\\\frac{n-3}2\end{pmatrix} =  \begin{pmatrix}n-1\\\frac{n-1}2\end{pmatrix}&\text{if}&n\ \ \text{is odd}\ \\
\end{array}
\right. \nn
\eeqa
We report the values of the coranks up to step $12$ in Table~\ref{table:crks}. Computing ratios yields:
\beqa
\frac{\text{Cork}(A^{(2)}_{2m+1})}{\text{Cork}(A^{(2)}_{2m})}=\frac{2m-1}{m}\ , \quad
\frac{\text{Cork}(A^{(2)}_{2m+2})}{\text{Cork}(A^{(2)}_{2m+1})}=2\  \nn
\eeqa
We thus have generated, with a recursive procedure, an infinite family of Kac--Moody algebras of strictly (exponentially) increasing coranks
. We stress that the Cartan matrix has more than one negative eigenvalue and that the associated Dynkin diagram is highly connected, since any vertex of the associated $(n+1)$-cube is connected to $n+1$ other vertices (three vertices are connected by a single line and the remaining $n-2$ by a double line).

\subsubsection{$L=1$ case}

We now turn to the case where there is a single link between cubes, \textit{i.e.} $L=1$. Eigenvalues connected by the recursion relation are equal modulo $2$, and the coefficients $a,b,c$ are easily identified considering the two characteristic polynomials $P_3$ and $P^{(1)}_4$:
\beqa
P_3(\lambda)=\lambda^{2} \left(\lambda +1\right)\left(\lambda -1\right)
\left(\lambda -2\right)^{4} \left(\lambda -3\right)\left(\lambda -4\right)^{2}
\left(\lambda -5\right) \nn
\eeqa
\beqa
P^{(1)}_4(\lambda)= \lambda^2 (\lambda +2) (\lambda +1)^2 (\lambda -1)^6 (\lambda -2)^2 (\lambda -3)^6 (\lambda -4)^2 (\lambda -5)^2 (\lambda -6) \nn
\eeqa
(Actually, the general recursion for this case already starts at the matrix $A_2$, for which $P_2(\lambda)=\lambda(\lambda-1)^2(\lambda-3)^2(\lambda-4)$.)
Following the same approach as when $L=2$, we obtain the degeneracies reported in Table~\ref{table:degFirstL1}.

\begin{table}[h!]
\centering
\begin{tabular}{c|cc}
$d^{(1)}_n(\lambda)$ & $\quad \lambda\ \  \text{even} \quad $ & $\quad  \lambda\ \  \text{odd} \quad  $ \\
\hline
\vspace{-8pt}& &\\
even $n$ & \quad $\begin{pmatrix}n-2\\\frac \lambda 2 +\frac{n-2}2\end{pmatrix}+\begin{pmatrix}n-2\\\frac \lambda 2 +\frac{n-6}2\end{pmatrix}\quad $ & $2 \begin{pmatrix}n-1\\\frac {\lambda-1} 2 +\frac{n-2}2\end{pmatrix}$ \\
\vspace{-8pt}& &\\
\hline
\vspace{-8pt}& &\\
odd $n$ & $2 \begin{pmatrix}n-1\\\frac {\lambda} 2 +\frac{n-3}2\end{pmatrix} $ & $\begin{pmatrix}n-2\\\frac {\lambda-1} 2 +\frac{n-1}2\end{pmatrix}
+ \begin{pmatrix}n-2\\\frac {\lambda-1} 2 +\frac{n-5}2\end{pmatrix}$\\
\end{tabular}
\caption{Values of the degeneracies $d^{(1)}_n(\lambda)$ for the first family of constructions when $L=1$.}
\label{table:degFirstL1}
\end{table}

Let us comment on the appearance of two separated binomial coefficients, focusing on the case of $d^{(1)}_{n=3}(\lambda =1\mod 2)$. We have that
\beqa
d^{(1)}_3(-1)=d^{(1)}_3(1)=d^{(1)}_3(3)=d^{(1)}_3(5)=1\nn
\eeqa
while, if only one binomial coefficient were sufficient,
\beqa
d^{(1)}_3(\lambda)=a\begin{pmatrix}3-b\\\frac{\lambda-r}{2}+\frac{3-k}{2}+c\end{pmatrix}\nn
\eeqa
However, the only sets of binomial coefficients $\left\{\begin{pmatrix}i\\j\end{pmatrix}, j\in \mathbb{N}\right\}$ where all values are equal are those when $i=0,1$, and they contain either one or two non-vanishing coefficients, not four. Consequently, we need to sum at least two binomial coefficients of the form \eqref{eq:degeneracyAnsatz}, with two sets of constants $a,b,c,k,r$. Two are sufficient, since we can write
\beqa
d^{(1)}_3(-1)=\begin{pmatrix}1\\0\end{pmatrix} \ , \qquad d^{(1)}_3(1)=\begin{pmatrix}1\\1\end{pmatrix} \ , \qquad d^{(1)}_3(3)=\begin{pmatrix}1\\0\end{pmatrix} \ , \qquad d^{(1)}_3(5)=\begin{pmatrix}1\\1\end{pmatrix}\nn
\eeqa
The need for two coefficients can also be traced back to the fact that $\lambda=2$ is not an eigenvalue of $A_2$, while $\lambda=1,3$ are.\\

A general formula for the corank again follows:
\beqa
\text{Cork}(A^{(1)}_{n})= \left\{
\begin{array}{cll}
\begin{pmatrix}n-2\\\frac{n-2}2\end{pmatrix}
+\begin{pmatrix}n-2\\\frac{n-6}2\end{pmatrix}&\text{if}&n\ \ \text{is even}\ \\
 2 \begin{pmatrix}n-1\\\frac{n-3}2\end{pmatrix}&\text{if}&n\ \ \text{is odd}\ \\
\end{array}
\right. \nn
\eeqa
We again report the values of the coranks up to step $12$ in Table~\ref{table:crks}. Computing ratios yields:
\beqa
\frac{\text{Cork}(A^{(1)}_{2m+1})}{\text{Cork}(A^{(1)}_{2m})}= \frac{2m(2m-1)}{m^2-m+1} > 4 \ \text{for $m > 2$} \ , \quad
\frac{\text{Cork}(A^{(1)}_{2m+2})}{\text{Cork}(A^{(1)}_{2m+1})}= \frac{m^2+m+1}{m(m+2)} < 1 \ \text{for $m > 1$}  \nn
\eeqa
We thus have constructed again, recursively, an infinite family of Kac--Moody algebras of high corank. However, the behavior of the corank along the recursion is different: rather than increasing monotonically, it oscillates because of $\text{Cork}(A^{(1)}_{2m+1}) > \text{Cork}(A^{(1)}_{2m})$ and yet $\text{Cork}(A^{(1)}_{2m+2}) < \text{Cork}(A^{(1)}_{2m+1})$. (The tendency is nonetheless again an exponential growth.) The Cartan matrices still have more than one negative eigenvalue and the Dynkin diagrams still are highly connected.
\\

We summarize the results of the two constructions in Table~\ref{table:crks}. While $L=2$ yields monotonically increasing coranks, $L=1$ sweeps through a greater variety of distinct corank values as it oscillates between steps.
\begin{table}[h]
\begin{center}
\begin{tabular}{c|c|c|c|c}
Step&Cartan&Size&Corank $L=1$&Corank $L=2$\\ \hline
1&$A^{(L)}_1$&3&0&0\\
2&$A^{(L)}_2$&6&1&1\\
3&$A^{(L)}_3$&12&2&2\\
4&$A^{(L)}_4$&24&2&4\\
5&$A^{(L)}_5$&48&8&6\\
6&$A^{(L)}_6$&96&7&12\\
7&$A^{(L)}_7$&192&30&20 \\
8&$A^{(L)}_8$&384&26&40\\
9&$A^{(L)}_9$&768&112&70 \\
10&$A^{(L)}_{10}$&1536&98&140\\
11&$A^{(L)}_{11}$&3072&420&252\\
12&$A^{(L)}_{12}$&6144&372&504\\
\end{tabular}
\end{center}
\caption{Values of the coranks for the first family of constructions when $n\leq 12$ and $L=1,2$.}
\label{table:crks}
\end{table}
\\

\begin{figure}[h!]
\begin{center}
\includegraphics[width=0.7\textwidth]{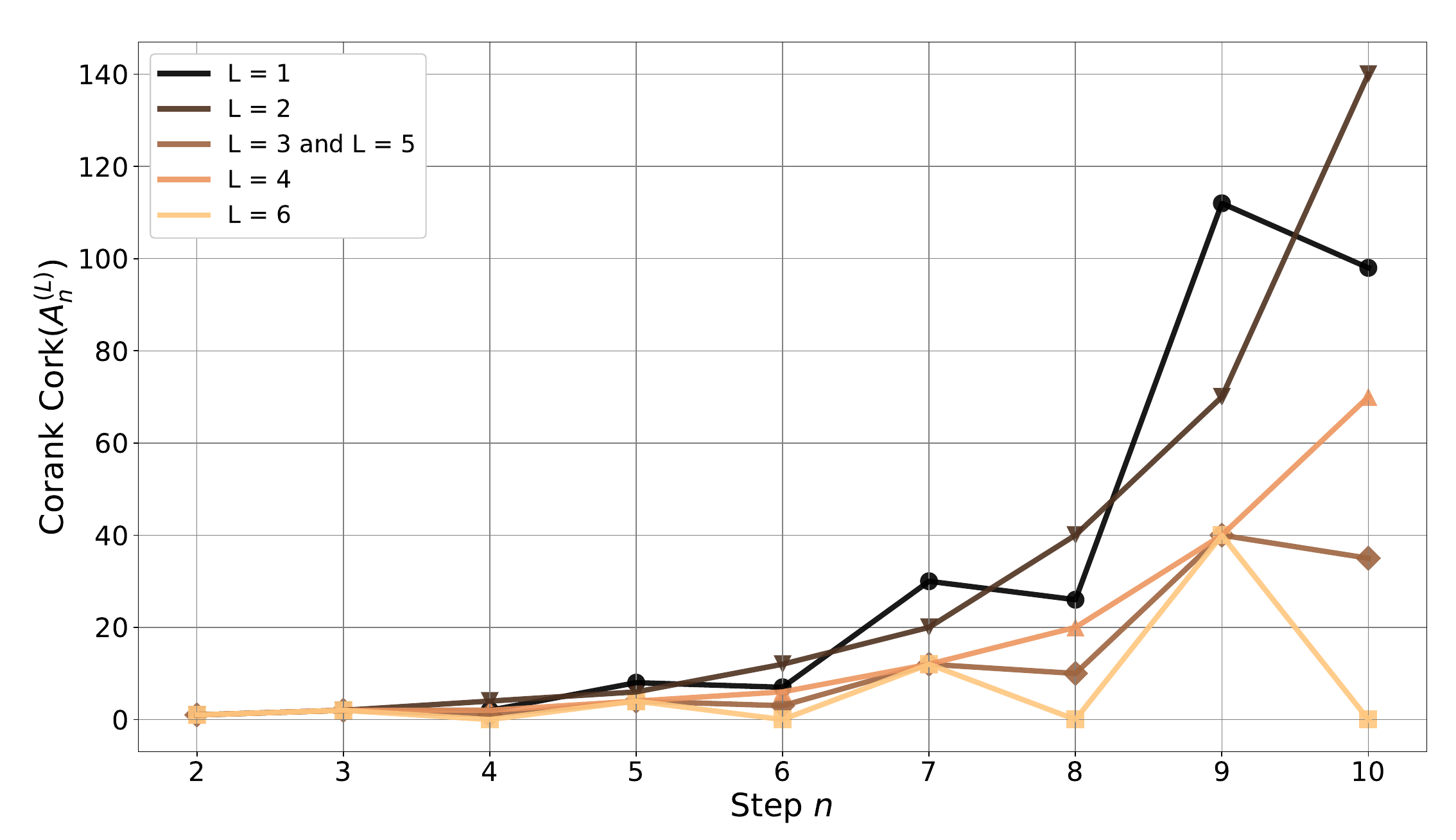}
\end{center}
\captionsetup{width=0.7\textwidth} 
\caption{Evolution of the coranks depending on the value of $L$. We observe two main trends, either a monotonic growth, or oscillations.}
\label{fig:crks}
\end{figure}

For the sake of completeness, we also summarize the results for more values of $L$ in Figure~\ref{fig:crks}. We observe different behaviors. First, for $L \geq 6$, the coranks become independent of $L$ and coincide at each step. This is due to the fact that the eigenvalues at step $3$ are $-1,0,1,2,3,4,5$, so $\lambda\not = 0\pmod L$ for any other $\lambda$. Therefore, the corank is simply given by (see Theorem~\ref{thm:first-family-all-L})
\beqa
2\begin{pmatrix}n-3\\\frac{n-3}{2}\end{pmatrix} \nn
\eeqa
for odd $n$ and $0$ for even $n$. Furthermore, since at step $3$ we have $\lambda\not = 0\pmod {2L}$ for any other $\lambda$, coranks agree for all $L\geq 3$ when $n$ is odd. When $n$ is even and $L<6$, there are two different trends. When $L$ is even, \textit{i.e.} for $L=2$ and $L=4$, the corank is strictly growing. For odd $L$, we see an alternating trend: between an odd step and an even step the corank decreases but between an even step and an odd step the corank increases. Finally, the coranks for $L=3$ are systematically lower than for $L=1$. All of these trends can be understood from the multiplicities at step $3$: they are larger for even eigenvalues than for odd ones, and among even eigenvalues $0$ has the smallest multiplicity. Therefore, in cases where only even eigenvalues of steps $m<n$ contribute to the corank at step $n$, which happens when $L$ is even, the corank can only increase with $n$. In cases when $L$ is odd and $n$ is even, the corank originates from the propagation through the recursion relation of the degeneracies of odd eigenvalues of step $3$, and is expected to be small. Finally, when $L=3$, the initial eigenvalues congruent to $0$ modulo $L$ are precisely $0$ and $3$, and therefore only the propagation through the recursion relation of the degeneracies of eigenvalues $0$ and $3$ of step $3$ can affect the corank. When $L=1$, $0=\lambda$ mod $1$ for all eigenvalues $\lambda$ at step $3$. The formulae we have found are coherent with Theorem~\ref{thm:first-family-all-L}.

\begin{remark}[Absolute versus relative growth]\label{rem:absolute-relative-growth-first}
The word ``high-corank'' refers here to absolute corank as a function of the recursive step: central-binomial estimates show that the coranks occurring in the formulae above grow exponentially in $n$ along the relevant subsequences. Since the size of the matrices in the first family is $3\cdot 2^{n-1}$, the ratio of the corank to the full matrix size is typically of order $n^{-1/2}$ rather than bounded away from zero. For example, in the $L=2$ even subsequence,
\[
        \Cork(A_{2m}^{(2)})=2\binom{2m-2}{m-1}
        \sim \frac{2^{2m-1}}{\sqrt{\pi(m-1)}}
        \qquad (m\to\infty)
\]
whereas the matrix size is $3\cdot 2^{2m-1}$.
\end{remark}

\subsection{Second family of constructions}

The first family of constructions used $n$-cubes with missing vertices and edges in order to explicitly generate the algebra. A different construction, presented here, avoids introducing such missing vertices and edges and instead builds the entire structure directly from regular $n$-cubes.\\

The construction starts at step $1$ with a zero-cube, which has one vertex and zero edge. The corresponding Cartan matrix is $C_1 = 2$ and corresponds to $\mathfrak{a}_1$. Then, very similarly to the previous approach, two parallel $n$-cubes are linked with $L$ links to construct an $(n+1)$-cube and its Cartan matrix. The general inductive formula for the Cartan matrix is the following: if the Cartan matrix associated to a $(n-1)$-cube at step $n$ is called $C^{(L)}_{n}$, we have at step $n+1$:
\beqa
C^{(L)}_{n+1} =\begin{pmatrix} C^{(L)}_{n}&-L \ \mathrm{Id}\\-L \ \mathrm{Id}&C^{(L)}_{n} \end{pmatrix}\ \nn
\eeqa
where $\operatorname{Id}$ is the identity matrix. (Notice the different use of $n$ with respect to the first family of constructions, where one had a $(n+1)$-cube at step $n$.) If $R^{(L)}_n$ is the characteristic polynomial of $C^{(L)}_n$, we have, as in the previous case,
\beqa
R^{(L)}_{n+1}(\lambda) = R^{(L)}_n(\lambda-L)\times R^{(L)}_n(\lambda+L) \nn
\eeqa

Since $R_1^{(L)}(\lambda)=\lambda-2$, Lemma~\ref{lem:doubling-spectrum} gives the whole spectrum immediately.

\begin{theorem}[Spectrum of the second family for any $L$]\label{thm:second-family-all-L}
For every $n\geq1$ and every positive integer $L$,
\begin{equation}\label{eq:second-spectrum-all-L}
        \Spec\bigl(C_n^{(L)}\bigr)
        =
        \left\{2+L(n-1-2j):0\leq j\leq n-1\right\}\nn
\end{equation}
where the eigenvalue corresponding to $j$ has multiplicity $\binom{n-1}{j}$. Equivalently
\begin{equation}\label{eq:second-degeneracy-all-L}
        d^{(L)}_n(\lambda)=
        \binom{n-1}{\frac{n-1}{2}+\frac{2-\lambda}{2L}}\nn
\end{equation}
Consequently,
\begin{equation}\label{eq:second-corank-all-L}
        \Cork\bigl(C_n^{(L)}\bigr)=
        \binom{n-1}{\frac{n-1}{2}+\frac{1}{L}}\nn
\end{equation}
In particular, among positive integers $L$, non-zero corank can occur only for $L=1$ or $L=2$.
\end{theorem}

\begin{proof}
At each of the $n-1$ recursive steps, the unique seed eigenvalue $2$ is shifted by either $+L$ or by $-L$. If exactly $j$ shifts are $-L$, the resulting eigenvalue is $2+L(n-1-2j)$, with multiplicity $\binom{n-1}{j}$. Setting this eigenvalue equal to zero gives $j=(n-1)/2+1/L$. For $j$ to be integral, given that $L$ is a positive integer, one must have $L=1$ or $L=2$.
\end{proof}


In the case $L=2$, only Cartan matrices at step $2m$ (associated to some $2m-1$-cubes) have a non-vanishing corank, while for $L=1$, only Cartan matrices at step $2m+1$ (associated to some $2m$-cubes) have a non-vanishing corank. Table~\ref{table:crks2} summarizes the values of the coranks for the two cases up to step $12$. A notable feature, unlike the previous family, is the systematic presence of zeros. The non-zero coranks do not oscillate but grow at the steps when they are non-vanishing. This follows from the fact that they all derive from the propagation of the degeneracy of the eigenvalue $2$ of step $1$ through the recursion relation.\\



\begin{table}[h]
\begin{center}
\begin{tabular}{c|c|c|c|c}

Step&Cartan&Size&Corank $L=1$&Corank $L=2$\\ \hline
1&$C^{(L)}_1$& 1&0&0\\
2&$C^{(L)}_2$&2&0&1\\
3&$C^{(L)}_3$&4&1&0\\
4&$C^{(L)}_4$&8&0&3\\
5&$C^{(L)}_5$&16&4&0\\
6&$C^{(L)}_6$&32&0&10\\
7&$C^{(L)}_7$&64&15&0\\
8&$C^{(L)}_8$&128&0&35 \\
9&$C^{(L)}_9$&256&56&0\\
10&$C^{(L)}_{10}$&512&0&126 \\
11&$C^{(L)}_{11}$&1024&210&0\\
12&$C^{(L)}_{12}$&2048&0&462\\
\end{tabular}
\end{center}
\caption{Values of the coranks for the second family of constructions when $n\leq 12$ and $L=1,2$.}
\label{table:crks2}
\end{table}

\begin{remark}[Asymptotics for the second family]
The size of $C_n^{(L)}$ is $2^{n-1}$. For $L=2$ and $n=2m$,
\[
        \Cork(C_{2m}^{(2)})=\binom{2m-1}{m-1}
        \sim \frac{2^{2m-1}}{\sqrt{\pi m}}
\]
For $L=1$ and $n=2m+1$,
\[
        \Cork(C_{2m+1}^{(1)})=\binom{2m}{m-1}
        \sim \frac{4^m}{\sqrt{\pi m}}
\]
Thus the non-zero coranks again grow exponentially in the step, while their relative size is asymptotically of order $n^{-1/2}$.
\end{remark}

\section{Further methods for obtaining high-corank Cartan matrices}\label{sec:further-methods}

This section records several constructions which are not used in the two explicit families above but which show that high corank is a broader phenomenon. The common principle is to make the zero eigenvalue of a generalized Cartan matrix occur with large multiplicity while preserving the Cartan sign conditions and indecomposability.

\subsection*{Spectral graph constructions}
Several graph-theoretic operations exist that create Cartan matrices whose eigenvalue $2$ has high multiplicity. This includes Cartesian products, graph covers, equitable partitions, and constructions from association schemes. The present paper uses a particularly rigid weighted hypercube-type operation, but many standard spectral operations can be adapted; see \cite{BrouwerHaemers,CvetkovicRowlinsonSimic,GodsilRoyle}.


\subsection*{Cartesian products and Kronecker sums}
If $G$ and $H$ are graphs, then the adjacency matrix of the Cartesian product $G\square H$ is the Kronecker sum
\[
        \Adj(G\square H)=\Adj(G)\otimes \mathrm{Id}+\mathrm{Id}\otimes\Adj(H)
\]
Hence its eigenvalues are sums $\lambda_i(G)+\lambda_j(H)$. Consequently,
\[
        2\mathrm{Id}-\Adj(G\square H)
\]
has corank equal to the number of pairs $(i,j)$, counted with multiplicity, such that
\[
        \lambda_i(G)+\lambda_j(H)=2
\]
This gives a flexible way to engineer high coranks. The second family above is essentially a weighted hypercube/Kronecker-sum construction: its binomial multiplicities are exactly the multiplicities produced by repeated two-point Cartesian factors.

\subsection*{Graph lifts and covers}
Graph lifts provide another systematic source of large graphs with controlled spectra. A lift contains old eigenvalues inherited from the base graph and new eigenvalues governed by the covering data. By choosing the covering permutations, or in the simplest case signs in a $2$-lift, one can attempt to preserve or create large eigenspaces at eigenvalue $2$. Iterating such lifts can therefore produce connected multigraphs $G$ for which $2\mathrm{Id}-\Adj(G)$ has high corank. This is conceptually close to the block doubling used here, but general lifts are less rigid because the identity matching between the two copies can be replaced by arbitrary permutations; compare the spectral use of $2$-lifts in \cite{BiluLinial}.

\subsection*{Block-matrix constructions with prescribed null-vectors}
One can also build high-corank generalized Cartan matrices directly by blocks. Suppose $A_1,\ldots,A_s$ are generalized Cartan matrices and one chooses off-diagonal blocks $B_{ij}$ with non-positive entries. If vectors $v_i\in\ker (A_i)$ are prescribed, then requiring
\[
        A_i v_i+\sum_{j\ne i}B_{ij}v_j=0
\]
for several independent choices of $(v_1,\ldots,v_s)$ produces a connected block Cartan matrix with a prescribed kernel. Such constructions are easiest when the seed null-vectors have zero or mixed-sign entries, because strictly positive affine null-vectors are constrained by the corank-one result of Section~\ref{sec:kernels}.

\subsection*{Gram matrices and reflection geometry}
A symmetric generalized Cartan matrix is also a Gram matrix of simple roots with
\[
        (\alpha_i,\alpha_i)=2,
        \qquad
        (\alpha_i,\alpha_j)\leq0 \quad (i\ne j)
\]
Thus one may work in an integral lattice with an indefinite symmetric bilinear form, choose many norm-$2$ vectors satisfying the non-positivity condition off the diagonal, and arrange that they span a vector space of dimension much smaller than their number. The resulting Gram matrix has large corank. This viewpoint connects the problem to Coxeter and reflection geometry, including Vinberg-type constructions of reflection groups in indefinite spaces \cite{Vinberg67}.

\subsection*{Symmetrizable non-symmetric variants}
Starting from a symmetric high-corank generalized Cartan matrix $B$, one may sometimes obtain non-symmetric symmetrizable matrices by diagonal rescaling. If $D$ is a positive diagonal rational matrix such that
\[
        A=D^{-1}B
\]
has integral entries and satisfies the generalized Cartan conditions, then $A$ is symmetrizable and has the same rank and corank as $B$. This produces examples outside the symmetric graph class without changing the kernel dimension.

\subsection*{Borcherds--Kac--Moody extensions}
If one relaxes the ordinary generalized Cartan matrix axioms and allows Borcherds--Cartan matrices, then imaginary simple roots can be introduced and high corank becomes easier to engineer. This is outside the strict framework of the present paper, but it gives a natural direction for comparison with generalized Kac--Moody algebras in the sense of Borcherds \cite{Borcherds88}.

\section{Conclusions and outlook}\label{sec:conclusion}

We have constructed two explicit recursive families of symmetric generalized Cartan matrices with rapidly growing corank. The construction is elementary but rather efficient: one starts from a seed Cartan matrix and repeatedly doubles it by the block operation
\[
        A\longmapsto
        \begin{pmatrix}
        A&-L\ \mathrm{Id}\\
        -L\ \mathrm{Id}&A
        \end{pmatrix}
\]
The operation has a transparent graph-theoretic interpretation and a completely controlled spectral effect. The eigenvalues of the new matrix are obtained by shifting the old eigenvalues by $\pm L$, and the corank is obtained by counting the old eigenvectors with eigenvalues $\pm L$. This explains both the existence of large kernels and the different behaviors observed for different values of the linking parameter $L$.

The main mathematical point is that these matrices are not affine-type degenerations. In the affine case, an indecomposable Cartan matrix has corank one and possesses a primitive null-vector with strictly positive entries. That vector defines the null root, and the corresponding Kac--Moody algebra admits the familiar central element and derivation of the loop-algebra realization. For the high-corank matrices considered here, the situation is qualitatively different. The kernel has dimension greater than one, and no non-zero kernel vector can have all entries of the same sign. Thus the associated integral radical vectors are isotropic elements of the root lattice, but they are not roots. The radical is multi-dimensional, the center in the derived Cartan subalgebra is multi-dimensional, and the minimal realization requires several independent derivation directions.

This observation suggests a possible interpretation of the corank of a generalized Cartan matrix as a measure of the number of independent algebraic null directions carried by the corresponding realization. In finite types, the corank is zero and there is no radical. In affine types, the corank is one and the single null direction admits a loop-theoretic interpretation. In the present high-corank setting, one obtains a multi-null structure. At a purely mathematical level, this means that the invariant bilinear form on the root lattice has a radical of dimension larger than one. At a physical or geometric level, it is tempting to interpret these directions as algebraic analogues of independent central charges, mutually orthogonal lightlike sectors, degenerating moduli, or simultaneous constraints. Such an interpretation is necessarily speculative at this stage, but it gives a useful language for comparing high-corank Kac--Moody algebras with better-known symmetry algebras in mathematical physics.\\

Several mathematical developments are natural. First, one should classify, or at least systematically organize, connected multigraphs $G$ for which the adjacency eigenvalue $2$ has large multiplicity. Since
\[
        \Cork(2\mathrm{Id}-\Adj(G))=\mult_{\Adj(G)}(2)
\]
this is precisely a spectral graph theory problem. Cartesian products, graph lifts, equitable partitions, association schemes and signed or weighted covers provide many variants of the present construction. Second, one should study the imaginary root system and root multiplicities of the resulting Kac--Moody algebras. High corank does not merely enlarge the center; it also changes the geometry of the isotropic cone in the root lattice, and this may have consequences for denominator identities, Weyl-group orbits, and automorphic correction problems. Third, it would be useful to determine the diagram automorphism groups of the recursive families rigorously, perhaps by combining graph-automorphism algorithms with an inductive proof. Fourth, one can pass from symmetric to symmetrizable non-symmetric Cartan matrices by integral diagonal rescalings, thereby preserving the rank and corank while producing non-simply-laced analogues.

A further direction concerns generalized, Borcherds and tensor-hierarchy-type extensions. If the condition $A_{ii}=2$ is relaxed, then generalized Kac--Moody algebras with imaginary simple roots enter the picture \cite{Borcherds88}. Such algebras are already natural in string theory, automorphic products and denominator formulae. The present construction remains inside the ordinary generalized Cartan framework, but its abundant isotropic radical directions suggest that high-corank ordinary Cartan matrices could provide singular limits or substructures of more general Borcherds--Kac--Moody systems. This is especially relevant whenever one wants many null or central directions without immediately introducing imaginary simple roots.

The possible relation with gravitational and supergravity symmetry algebras deserves special attention. Infinite-dimensional Kac--Moody algebras occur in the hidden-symmetry approach to gravity and supergravity. The affine algebra $E_9$ appears after reduction of maximal supergravity to two dimensions; the hyperbolic algebra $E_{10}$ appears in the cosmological-billiard and small-tension-expansion approach to eleven-dimensional supergravity; and the very-extended Lorentzian algebra $E_{11}$ has been proposed by West as an organizing algebra for M-theory and its brane charges \cite{Julia85,DamourHenneauxNicolai2002,DamourHenneauxNicolai2003,WestE11,KleinschmidtNicolai2004,KleinschmidtNicolai2005}. These algebras are indefinite, but their Cartan matrices are not high-corank objects in the sense studied here. Their physical relevance comes mainly from their Lorentzian or hyperbolic root geometry, their level decompositions with respect to finite-dimensional regular subalgebras, and their ability to organize towers of supergravity fields and dual fields.

High-corank Kac--Moody algebras are therefore complementary to $E_{10}$, $E_{11}$ and their very-extended relatives. Whereas $E_{10}$ and $E_{11}$ emphasize a large but essentially non-degenerate indefinite root lattice, the present families emphasize a large radical. In a speculative physical reading, replacing a Lorentzian non-degenerate Cartan matrix by a highly singular one could correspond to a degenerate limit in which several independent null directions become simultaneously central. Such a situation might be relevant for multi-parameter tensionless limits, for contractions of duality algebras, for sectors with many mutually compatible brane charges, or for algebraic models in which several Hamiltonian or gauge constraints become central at once. One should not overstate this point: no direct supergravity model is constructed in the present paper. Nevertheless, the comparison indicates a concrete question for future work: can high-corank generalized Cartan matrices arise as contraction limits, degenerations, quotients, or auxiliary symmetry algebras associated with known $E_{10}$/$E_{11}$ level decompositions?

Another possible application concerns compactification and duality. In toroidal compactifications, hidden symmetry groups grow as the dimension is lowered, and their affine or hyperbolic extensions encode increasingly non-local degrees of freedom. A high-corank algebra could provide an algebraic model for situations where more than one independent affine-like or null direction is present. This may be relevant to reductions with several distinguished circles, to doubled or exceptional-geometric descriptions with multiple section constraints, or to degenerations of moduli spaces in which several cycles collapse simultaneously. The corank would then count, not a spacetime dimension directly, but the number of independent radical directions in the algebraic symmetry data. This viewpoint could be tested by comparing the radicals of the Cartan matrices with charge lattices, central extensions, or constraint algebras in concrete compactification models.

From the representation-theoretic side, an important problem is to understand integrable highest-weight modules for the algebras constructed here. The presence of several central elements means that highest weights carry several independent levels. In affine theory, the level is a fundamental invariant of a representation; in the high-corank case, the level becomes a vector in the dual of the center. This raises natural questions about unitarity, characters, denominator formulae, modularity, and possible multi-variable generalizations of affine characters. Even for the explicit corank-two example studied above, the structure of imaginary roots and their multiplicities should already contain information not present in ordinary affine theory.

Finally, the graph-theoretic formulation suggests computational projects. One can enumerate connected multigraphs up to a given size whose Cartan matrix $2\mathrm{Id}-\Adj(G)$ has prescribed corank, search for extremal examples maximizing corank under degree or regularity constraints, study random graph models conditioned to have adjacency eigenvalue $2$, and compare the resulting Kac--Moody algebras with the recursive hypercube families. This would produce a bridge between computational spectral graph theory and the construction of new indefinite Kac--Moody algebras.\\

In summary, the present work provides explicit high-corank examples and a general spectral mechanism for producing many more. The most immediate mathematical tasks are to refine the classification of such constructions, to compute the associated root multiplicities and automorphism groups, and to understand their representation theory. The most speculative but potentially fruitful physical task is to determine whether the corank can acquire an invariant meaning in models with several independent central, null, constrained or tensionless sectors, in analogy with the way the unique affine null direction has a clear meaning in loop-algebra and two-dimensional integrable contexts.\\

\section*{Acknowledgments}
SB and VS are supported by the Interdisciplinary Thematic Institute QMat, as part of the ITI 2021-2028 program of the University of Strasbourg, CNRS and Inserm, and was supported by IdEx Unistra (ANR-10-IDEX-0002), SFRI STRAT' US project (ANR-20-SFRI-0012), and EUR QMAT (ANR-17-EURE-0024). Both these fundings were granted under the framework of the French Investments for the Future Program. QB is supported by an IdEx funding ``Attractivity 2025'' of the University of Strasbourg. AM's work is supported by COST Action CaLISTA CA21109 funded by COST (European Cooperation in Science and Technology).

\newpage

\appendix

\section{The Lie algebra associated to $A_3$}\label{app:A4}

As we have seen, the matrix $A_3$ has corank two. The vectors
\beqa
v_3&=&\Big(\begin{array}{cccccccccccc}
1 & 0&-1&-1& 0& 1& 1& 0& -1& -1& 0&1
\end{array}\Big)^t \nn\\
w_3&=&\Big(\begin{array}{cccccccccccc}
0 &1& 1& 0& -1& -1& 0& 1& 1& 0& -1& -1
\end{array}\Big)^t \nn
\eeqa
generate a basis of its kernel.\\

Let $h_i$ be the generators of the Cartan subalgebra.
The two elements
\begin{eqnarray*}
k_{1} &=&v_3^i h_i= h_{1}-h_{3}-h_{4}+h_{6}+h_{7}-h_{9}-h_{10}+h_{12}\ \\
k_{2} &=&w_3^ih_i=h_{2}+h_3-h_5-h_{6}+h_{8}+h_{9}-h_{11}-h_{12}\
\end{eqnarray*}%
\bigskip
are central.

Let $\alpha_i$ be the simple roots. Since the matrix $A_3$ is symmetric, we can define the scalar product on the root-space
\begin{equation*}
(\alpha _{i},\alpha _{j})=(A_{3})_{ij}
\end{equation*}%
It follows that $A_3$ admits two zero-norm vectors:
\beqa
\label{eq:del}
\delta _{1} &=&\alpha_{1}-\alpha_{3}-\alpha_{4}+\alpha_{6}+\alpha_{7}-\alpha_{9}-\alpha_{10}+\alpha_{12}\nn\\
 \delta _{2} &=&\alpha_{2}+\alpha_3-\alpha_5-\alpha_{6}+\alpha_{8}+\alpha_{9}-\alpha_{11}-\alpha_{12}\nn
\eeqa%
Such two vectors span a $2-$dimensional isotropic subspace in $12$ dimensions and we have for any
roots $\alpha: (\alpha,\delta_i)=0$, $i=1,2$.\\

Let $(e_i,f_i,h_i), i=1,\cdots,12$ be the $\mathfrak{sl}(2,\mathbb C)$ triplets associated to the simple roots. Since the corank of $A_{3}$ is two, one must add two derivations (see Definition \ref{theo:Weyl}). Extend $\mathfrak{h}'=\{h_1,\cdots, h_{12}\}$ to $\mathfrak{h}$ where $\mathfrak{h}=\{h_1,\cdots, h_{12},d_1,d_2\}$. We now compute $[d_a,e_i]$ and $[d_a,f_i]$.
To proceed we write
\beqa
A=\begin{pmatrix}A'&B\\C&D\end{pmatrix}\nn
\eeqa
with
\beqa
\begin{array}{cccccc}
A'&=&\left(\begin{array}{rrrrrrrrrr}
2 & -1 & 0 & 0 & 0 & -1 & -1 & 0 & 0 & 0
\\
 -1 & 2 & -1 & 0 & 0 & 0 & 0 & -1 & 0 & 0
\\
 0 & -1 & 2 & -1 & 0 & 0 & 0 & 0 & -1 & 0
\\
 0 & 0 & -1 & 2 & -1 & 0 & 0 & 0 & 0 & -1
\\
 0 & 0 & 0 & -1 & 2 & -1 & 0 & 0 & 0 & 0
\\
 -1 & 0 & 0 & 0 & -1 & 2 & 0 & 0 & 0 & 0
\\
 -1 & 0 & 0 & 0 & 0 & 0 & 2 & -1 & 0 & 0
\\
 0 & -1 & 0 & 0 & 0 & 0 & -1 & 2 & -1 & 0
\\
 0 & 0 & -1 & 0 & 0 & 0 & 0 & -1 & 2 & -1
\\
 0 & 0 & 0 & -1 & 0 & 0 & 0 & 0 & -1 & 2
\end{array}\right) \ , &
B&=&\left(\begin{array}{rr}
0 & 0
\\
 0 & 0
\\
 0 & 0
\\
 0 & 0
\\
 -1 & 0
\\
 0 & -1
\\
 0 & -1
\\
 0 & 0
\\
 0 & 0
\\
 -1 & 0
\end{array}\right)\nn\\ \\
C&=&\left(\begin{array}{rrrrrrrrrr}
0 & 0 & 0 & 0 & -1 & 0 & 0 & 0 & 0 & -1
\\
 0 & 0 & 0 & 0 & 0 & -1 & -1 & 0 & 0 & 0
\end{array}\right)\ , &
D&=&\left(\begin{array}{cc}
2 & -1
\\
 -1 & 2
\end{array}\right)
\end{array}\nn
\eeqa
Note that $B=C^t$. The matrix $A'$ is non-singular ($\det A'=-80$).
We now associate to $A$ a minimal realisation which is the $14\times14$
matrix \cite{Kac2, Moody} (see also the pedagogical references \cite{Ca, Mar}):
\beqa
\mathcal{C}=\begin{pmatrix} A'&B&{\bf 0}\\ C&D&\mathrm{Id}\\{\bf 0}&\mathrm{Id}&{\bf 0}
\end{pmatrix}\nn
\eeqa
with the notations of above. We have
\beqa
\big[d_1,e_i\big]= \delta_{i,11} e_i \ , \qquad
\big[d_2,e_i\big]=  \delta_{i,12} e_i \ ,\qquad \big[d_1,f_i\big]= - \delta_{i,11} f_i \ , \qquad
\big[d_2,f_i\big]= - \delta_{i,12} f_i \nn
\eeqa
All other realisations obtained by permuting rows and columns of $A_3$ lead to non-canonically isomorphic algebras.\\

The particular choice above singles out the last two simple-root labels when adjoining the two derivations. A different invertible $10\times10$ principal submatrix, or a different ordering of the simple roots, gives an isomorphic minimal realization but changes the displayed coordinates of the derivations and fundamental weights. The intrinsic data are the Cartan matrix, the two-dimensional center spanned by $k_1,k_2$, and the two-dimensional radical in the root lattice spanned by $\delta_1,\delta_2$.

Since the Cartan matrix $A_3$ is of corank two, the previous simple roots $\alpha_1,\cdots,\alpha_{12}$ are extended to $\alpha'_1=\alpha_1,\cdots,\alpha'_{12}=\alpha_{12}, \alpha'_{13}=\beta_1, \alpha'_{14} = \beta_{2}$ with two derivations so that $\alpha'_j(h'_i)=\mathcal{C}_{ij}$. Furthermore, the extended root space is endowed with the scalar product defined by
 \beqa
(\alpha'_i,\alpha'_j)= \mathcal{C}_{ij}  \nn
\eeqa
(see {\it e.g.} \cite{Ca, Mar}).
The fundamental weights $\mu'^1,\cdots,\mu'^{14}$ are as usual defined by
$(\mu'^i,\alpha'_j)=\delta^i{}_j$ and given by
\beqa
\mu'^i = (\mathcal{C}^{-1})^{ij}\alpha'_j \nn
\eeqa
We obtain the following weights:
{\small
\beqa
&&\mu'^1=\mu^1=
\frac{11}{10} \alpha_{1}+
\frac 1{5}\alpha_{2}-
\frac 7{10} \alpha_{3}-
\frac{17}{20} \alpha_{4}-
\frac 1{5}\alpha_{5}+
\frac{9}{20} \alpha_{6}+
\frac{11}{20} \alpha_{7}-
\frac{3}4 \alpha_{9}-
\frac{4}5 \alpha_{10}-\beta_{1}+\beta_{2}\nn\\
&&\mu'^2=\mu^2=
\frac 1 5\alpha_{1}+\frac{2} 5 \alpha_{2}-\frac 2 5  \alpha_{3}
-\frac 7{10} \alpha_{4}-\frac 2 5 \alpha_{5}\
-\frac 1{10}\alpha_{6}+\frac 1{10}\alpha_{7}-\frac 12 \alpha_{9}
-\frac 3 5 \alpha_{10}-\beta_{1}\nn\\
&&\mu'^3=\mu^3=
-\frac 7{10} \alpha_{1}-\frac 25 \alpha_{2}
+\frac 25 \alpha_{3}+\frac 1 5 \alpha_{4}
-\frac 1 {10} \alpha_{5}-\frac 2 5  \alpha_{6}-\frac 3 5  \alpha_{7}
-\frac 1 2 \alpha_{8}+\frac 1 {10} \alpha_{10}-\beta_{2}\nn\\
&&\mu'^4=\mu^4=
-\frac {17}{20} \alpha_{1} -\frac 7{10} \alpha_{2}
+\frac 1 5 \alpha_{3}+\frac{11}{10} \alpha_{4}
+\frac 9{20} \alpha_{5}-\frac 1 5 \alpha_{6}
-\frac 4 5  \alpha_{7}-\frac 3 4 \alpha_{8}+\frac{11}{20} \alpha_{10}
+\beta_{1}-\beta_{2}\nn \\
&&\mu'^5=\mu^5=
-\frac 1 5\alpha_{1}-\frac 2 5  \alpha_{2}-\frac 1 {10} \alpha_{3}
+\frac 9 {20}  \alpha_{4}+\frac 9{10} \alpha_{5}
+\frac 7 {20} \alpha_{6}-\frac 7{20}  \alpha_{7}-\frac 1 2 \alpha_{8}
-\frac 1 4 \alpha_{9} +\frac 1 {10} \alpha_{10}+\beta_{1}\nn\\
&&\mu'^6=\mu^6=
\frac 9{20} \alpha_{1}-\frac 1 {10}\alpha_{2}-\frac 2 5  \alpha_{3}-
\frac 1 5 \alpha_{4}+\frac 7{20} \alpha_{5}+\frac 9{10} \alpha_{6}
+\frac 1 {10}\alpha_{7}-\frac 1 4 \alpha_{8}-\frac 1 2 \alpha_{9}-
\frac 7{20} \alpha_{10}+\beta_{2}\nn\\
&&\mu'^7=\mu^7=
\frac{11}{20} \alpha_{1}+\frac 1 {10} \alpha_{2}-\frac 3 5  \alpha_{3}-
\frac 4 5  \alpha_{4}-\frac 7 {20} \alpha_{5}+\frac 1 {10}\alpha_{6}+
\frac 9{10} \alpha_{7}+\frac 1 4 \alpha_{8}-\frac 1 2 \alpha_{9}-
\frac{13}{20} \alpha_{10}-\beta_{1}+\beta_{2}\nn\\
&&\mu'^8=\mu^8=
-\frac 1 2\alpha_{3}-\frac 3 4 \alpha_{4}-\frac 1 2 \alpha_{5}-
\frac1 4 \alpha_{6}+\frac 1 4 \alpha_{7}+\frac 1 2 \alpha_{8}-
\frac 1 4 \alpha_{9}-\frac 1 2 \alpha_{10}-\beta_{1}\nn\\
&&\mu'^9=\mu^9=
-\frac 3 4 \alpha_{1}-\frac 1 2 \alpha_{2}-\frac 1 4 \alpha_{5}-
\frac 1 2 \alpha_{6}-\frac1 2 \alpha_{7}-\frac 1 4 \alpha_{8}+
\frac1 2 \alpha_{9}+\frac 1 4 \alpha_{10}-\beta_{2}\nn\\
&&\mu'^{10}=\mu^{10}=
-\frac 4 5 \alpha_{1}-\frac 3 5  \alpha_{2}+\frac 1 {10}\alpha_{3}+
\frac{11}{20} \alpha_{4}+\frac 1{10} \alpha_{5}-\frac7 {20} \alpha_{6}-
\frac{13}{20} \alpha_{7}-\frac 1 2 \alpha_{8}+\frac 1 4  \alpha_{9}+
\frac9 {10} \alpha_{10}+\beta_{1}-\beta_{2}
\nn\\
&&\mu'^{11}=\mu^{11}=\beta_1\nn\\
&&\mu'^{12}=\mu^{12}=\beta_2\nn\\
&&\mu'^{13}=-\delta_1-\delta_2=
-\alpha_{1}-\alpha_{2}+\alpha_{4}+\alpha_{5}-\alpha_{7}-\alpha_{8}+\alpha_{10}+\alpha_{11}\nn\\
&&\mu'^{14}=\delta_1=
\alpha_{1}-\alpha_{3}-\alpha_{4}+\alpha_{6}+\alpha_{7}-\alpha_{9}-\alpha_{10}+\alpha_{12}\nn
\eeqa
We have the following symmetries of the simple roots/fundamental weights:
\beqa\label{def:sym}
&\Big[\alpha_1 \leftrightarrow \alpha_4\ , \ \
\alpha_2 \leftrightarrow \alpha_3\ , \ \
\alpha_5 \leftrightarrow \alpha_6\ , \ \
\alpha_{1+6} \leftrightarrow \alpha_{4+6}\ , \ \
\alpha_{2+6} \leftrightarrow \alpha_{3+6}\ , \ \
\alpha_{5+6} \leftrightarrow \alpha_{6+6}\ , \ \
\alpha'_{13} \leftrightarrow \alpha'_{14}\Big] \nn\\
&\&\\
&\Big[\mu^1 \leftrightarrow \mu^4\ , \ \
\mu^2 \leftrightarrow \mu^3\ , \ \
\mu^6 \leftrightarrow \mu^7\ , \ \
\mu^{1+6} \leftrightarrow \mu^{4+6}\ , \ \
\mu^{2+6} \leftrightarrow \mu^{3+6}\ , \ \
\mu^{6+6} \leftrightarrow \mu^{7+6}\ , \ \
\mu'^{13} \leftrightarrow \mu'^{14}\Big]\nn
\eeqa}

\section{Diagram automorphisms of the constructed $A_n^{(L)}$ matrices}\label{app:OutAut}

Looking at the kernel of $A_3$ derived in Appendix~\ref{app:A4}, one observes that the palindromes of $v_3$ and $w_3$ ({\it i.e.} those vectors read from right to left) still are in the kernel. This property comes from the fact that the matrix $A_3$ is the same if you rotate it by $180^{\circ}$ about its center ($A^{\rm rot}_{ij}=A_{n+1-i,n+1-j}$ where $n$ is the size of the matrix). This rotational invariance property is conserved by the constructions we are studying: if $A$ is rotation-invariant, then
\beqa
A^{(L)}_{n+1} =\begin{pmatrix} A^{(L)}_{n}&-L \ \mathrm{Id}\\-L \ \mathrm{Id}&A^{(L)}_{n} \end{pmatrix}\ \nn
\eeqa
is also rotation-invariant. From this, we deduce that at every step the kernel of the matrix is invariant by the palindrome action.\\

The property that the Cartan matrix is rotation-invariant translates on the Dynkin diagram by the fact that it is also invariant under rotation of $180^{\circ}$ about its center (see the aztec pyramid in Section~\ref{pyramid}). Strictly speaking, these are diagram automorphisms of the generalized Cartan matrix, not outer automorphisms of the Kac--Moody algebra until one has passed through the standard map from diagram automorphisms to algebra automorphisms. We therefore formulate the following discussion at the level of diagram automorphisms.\\

Consider first the one at step two in the first family of constructions, so the diagram of $\widehat{\mathfrak{a}}_5$. This diagram is equivalent to a hexagon so the group of symmetry is the dihedral group of $12$ elements $D_6$. At the next step (the aztec pyramid), the diagram is made of two copies of $\widehat{\mathfrak{a}}_5$ where similar vertices are linked. Because of this, we cannot mix the two copies: permutation can only act within a single one, or exchange the two copies entirely. In the first option, if one applies an element of the outer automorphism to one of the copies, we need to apply the same element to the second copy. The group of automorphisms associated to this is the diagonal group which we will denote by $\text{Diag}(D_6,D_6)$. For the second option, we denote by $P$ the reflection around the vertical axis. If we compose this reflection by an element of the previous group, we obtain an outer automorphism that exchanges both copies. From this we can conjecture that the group of outer automorphisms of the aztec pyramid is generated by $\text{Diag}(D_6,D_6)$ and $P\circ\text{Diag}(D_6,D_6)$. We will denote it $\langle\text{Diag}(D_6,D_6),P\circ\text{Diag}(D_6,D_6)\rangle$.
\\

The following recursive description should be regarded as a conjectural structural description unless one performs a graph-automorphism computation, or gives an independent proof that no additional automorphisms occur at higher steps. In modern language, the expected structure is a diagonal copy of the previous automorphism group, together with the involution exchanging the two matched copies.

By the same reasoning, if we denote by $\text{Out}_n$ the group of diagram automorphisms at step $n$, we conjecture the inductive relation:
$$\text{Out}_{n+1}=\langle\text{Diag}(\text{Out}_n,\text{Out}_n),P\circ\text{Diag}(\text{Out}_n,\text{Out}_n)\rangle$$

Let us note that the symmetry we have made explicit in \eqref{def:sym} corresponds to the element of the automorphism group $\text{Out}_{3}$ of the form $\text{Diag}(\sigma,\sigma)$ where $\sigma$ is the outer automorphism of $\widehat{\mathfrak{a}}_5$ consisting in the composition of the reflection with respect to the $(1,4)$-axis and the $180^{\circ}$ rotation.

\newpage

\bibliographystyle{utphys}
\bibliography{ref_final}
\end{document}